\RequirePackage{fix-cm}
\documentclass[smallextended,numbook,runningheads]{svjour3}
\smartqed  
\usepackage{graphicx}
\usepackage{amsmath,amssymb}
\usepackage{graphicx,graphics,color}
\usepackage{booktabs,subfig} 

\journalname{...}
\date{ \phantom{b} \vspace{45mm}\phantom{e}}
\def\makeheadbox{\hspace{-4mm}Version of \today}

\usepackage{color}

\def\half{\frac{1}{2}}

\newcommand\bfd{{\mathbf d}}
\newcommand\bfe{{\mathbf e}}
\newcommand\bff{{\mathbf f}}
\newcommand\bfg{{\mathbf g}}

\newcommand\bfu{{\mathbf u}}
\newcommand\bfv{{\mathbf v}}
\newcommand\bfw{{\mathbf w}}
\newcommand\bfx{{\mathbf x}}
\newcommand\bfy{{\mathbf y}}
\newcommand\bfz{{\mathbf z}}
\newcommand\bfA{{\mathbf A}}

\newcommand\bfK{{\mathbf K}}
\newcommand\bfM{{\mathbf M}}

\newcommand\bfzero{{\mathbf 0}}

\newcommand\andquad{\quad\hbox{ and }\quad}
\newcommand\for{\quad\hbox{ for }\quad}

\renewcommand\d{\hbox{\rm d}}

\newcommand{\Ga}{\Gamma}

\newcommand{\laplace}{\Delta}
\newcommand{\lb}{\laplace_{\surface}}
\newcommand{\nbg}{\nabla_{\Gamma}}
\newcommand{\nbgh}{\nabla_{\Gamma_h}}
\newcommand{\mat}{\partial^{\bullet}}

\DeclareMathOperator{\diff}{\frac{\d}{\d t}}

\newcommand{\inv}{^{-1}}

\newcommand{\nb}{\nabla}

\newcommand{\pa}{\partial}
\newcommand{\R}{\mathbb{R}}

\newcommand{\spn}{\textnormal{span}}

\def \to {\rightarrow}
\newcommand{\vphi}{\varphi}

\usepackage{marginnote}

\newcommand\id{\mathrm{id}}
\newcommand\surface{\Gamma}
\newcommand\normal{\nu}
\newcommand\enodes\xs
\newcommand\nnodes\bfx
\newcommand\rnodes\xs

\newcommand\regmass\bfK
\newcommand\mass\bfM
\newcommand\stiff\bfA

\newcommand\dfdt[2]{\frac{\dd #1}{\dd #2}}

\newcommand\dd{\mathrm{d}}

\newcommand{\Xs}{X_h^\ast}
\newcommand{\Vs}{v_h^\ast}

\newcommand{\vhat}{\hat v_h}

\newcommand{\normM}[1]{\| #1 \|_{\bfM(\xs)}}
\newcommand{\normA}[1]{\| #1 \|_{\bfA(\xs)}}
\newcommand{\normK}[1]{\| #1\|_{\bfK(\xs)}}
\newcommand{\us}{\bfu^\ast}
\newcommand{\vs}{\bfv^\ast}
\newcommand{\xs}{\bfx^\ast}

\newcommand{\dotxs}{\dot\bfx^\ast}
\newcommand{\eu}{\bfe_\bfu}
\newcommand{\ev}{\bfe_\bfv}
\newcommand{\ex}{\bfe_\bfx}

\newcommand{\dotex}{\dot\bfe_\bfx}
\newcommand{\du}{\bfd_\bfu}
\newcommand{\dv}{\bfd_\bfv}

\newcommand\dotarg{\:\cdot\:}

\newcommand{\QED}{\hfill $\square$}

\newcommand{\bcl}{\color{black}}
\newcommand{\ecl}{\color{black}}

\newcommand{\bbk}{\color{black}}
\newcommand{\ebk}{\color{black}}

\begin{document}

\title{Convergence of finite elements on an evolving surface\\
driven by diffusion on the surface}

\titlerunning{Convergence of finite elements on a solution-driven evolving surface}        

\author{Bal\'{a}zs~Kov\'{a}cs \and
        Buyang~Li \and
        Christian~Lubich \and
        Christian~A.~Power~Guerra
}

\authorrunning{B.~Kov\'{a}cs, B.~Li, Ch.~Lubich and C.A.~Power Guerra} 

\institute{B. Kov\'{a}cs, Ch. Lubich and C. A. Power Guerra \at
              Mathematisches Institut, Universit\"at T\"{u}bingen,\\
              Auf der Morgenstelle 10, 72076 T\"{u}bingen, Germany \\
              \email{\{kovacs,lubich,power\}@na.uni-tuebingen.de}
           \and
           B. Li \at
              Department of Applied Mathematics, Hong Kong Polytechnic University,\\
              Kowloon, Hong Kong \\
              \email{buyang.li@polyu.edu.hk}
}

\date{}

\maketitle

\begin{abstract}
 For a parabolic surface partial differential equation coupled to surface evolution, convergence of the spatial semidiscretization is studied in this paper. The velocity of the evolving surface is not given explicitly, but depends on the solution of the parabolic equation on the surface.  Various velocity laws are considered: elliptic regularization of a direct pointwise coupling, a regularized mean curvature flow and a dynamic velocity law. A novel stability and convergence analysis for evolving surface finite elements for the coupled problem of surface diffusion and surface evolution is developed.
 \bcl The stability analysis works with the matrix-vector formulation of the method and does not use geometric arguments. The geometry enters only into the consistency estimates. \ecl
Numerical experiments complement the theoretical results.
\end{abstract}
\section{Introduction}
Starting from a paper by Dziuk and Elliott \cite{DziukElliott_ESFEM}, much insight into the stability and convergence properties of finite elements on evolving surfaces has been obtained by studying a linear parabolic equation on a given moving closed surface $\Gamma(t)$. The strong formulation of this model problem is to find a solution $u(x,t)$ (for $x\in \Gamma(t)$ and $0 \le t \le T$) with given initial data $u(x,0)=u_0(x)$ to the linear partial differential equation
$$
    \mat u(x,t) + u(x,t) \nb_{\Ga(t)} \cdot v(x,t) - \laplace_{\Ga(t)} u(x,t)  = 0, \qquad x\in \Gamma(t),\ 0 < t \le T,
$$
where  $\mat$ denotes the material time derivative,  $\laplace_{\Ga(t)}$ is the Laplace--Beltrami operator on the surface, and $\nb_{\Ga(t)} \cdot v$ is the tangential divergence of the {\it given} velocity $v$ of the surface. We refer to \cite{DziukElliott_acta} for an excellent review article (up to 2012) on the numerical analysis of this and related problems. Optimal-order $L^2$ error bounds for piecewise linear finite elements are shown in \cite{DziukElliott_L2} and maximum-norm error bounds in \cite{KovacsPower_max}. Stability and convergence of full discretizations obtained by combining the evolving surface finite element method (ESFEM) with various time discretizations are shown in \cite{DziukElliott_fulldiscr,DziukLubichMansour_rksurf,LubichMansourVenkataraman_bdsurf}.  Convergence of semi- and full discretizations using high-order evolving surface finite elements is studied in \cite{highorder}.
Arbitrary Euler--Lagrangian (ALE) variants of the ESFEM method for this equation are studied
in \cite{ElliottStyles_ALEnumerics,ElliottVenkataraman_ALE,KovacsPower_ALEdiscrete}.
Convergence properties of the ESFEM and of full discretizations for quasilinear parabolic equations on prescribed moving surfaces are studied in \cite{KovacsPower_quasilinear}.

Beyond the above model problem, there is considerable interest in cases where the velocity of the evolving surface is \emph{not given explicitly}, but depends on the solution $u$ of the parabolic equation; see, e.g., \cite{BEM,CGG,ElliottStyles_ALEnumerics,Fife} for physical and biological models where such situations arise. Contrary to the case of surfaces with prescribed motion, there exists so far no numerical analysis for solution-driven surfaces in $\R^3$, to the best of our knowledge. 

For the case of evolving {\it curves} in $\R^2$, there are recent preprints by Pozzi \& Stinner \cite{PozziStinner_curve} and Barrett, Deckelnick \& Styles \cite{BDS}, who couple the curve-shortening flow with diffusion on the curve and study the convergence of  finite element discretizations without and with a tangential part in the discrete velocity, respectively.
The analogous problem for two- or higher-dimensional surfaces would be to couple mean curvature flow with diffusion on the surface.
Studying the convergence of finite elements for these coupled problems, however, remains illusive as long as the convergence of ESFEM for mean curvature flow of closed surfaces is not understood. This has remained an open problem since Dziuk's formulation of such a numerical method for mean curvature flow in his 1990 paper \cite{Dziuk90}.

In this paper we consider  different velocity laws for coupling the surface motion with  the  diffusion on the surface. Conceivably the  simplest velocity law would be to prescribe the normal velocity at any surface point as a function of the solution value and possibly its tangential gradient at this point:
$v(x,t) = g(u(x,t), \nabla_{\Ga(t)}u(x,t))\, \nu_{\Ga(t)}(x)$ for $x\in \Ga(t)$, where
$\nu_{\Ga(t)}(x)$ denotes the outer normal vector and $g$ is  a given smooth scalar-valued function. This does, however, not appear to lead to a well-posed problem, and in fact we found no mention of this seemingly obvious choice in the literature. Here we study instead a {\it regularized velocity law}:
$$
v(x,t) - \alpha \laplace_{\Ga(t)} v(x,t)= g\bigl(u(x,t), \nabla_{\Ga(t)}u(x,t)\bigr)\, \nu_{\Ga(t)}(x),
\qquad x\in \Ga(t),
$$
with a fixed regularization parameter $\alpha>0$. This elliptic regularization will turn out to permit us to give a complete stability and convergence analysis of the ESFEM semidiscretization, for finite elements of polynomial degree at least two. The case of linear finite elements is left open in the theory of this paper, but will be considered in our numerical experiments. The stability and convergence results  can be extended to full discretizations with linearly implicit backward difference time-stepping, as we plan to show in later work.

Our approach also applies to the ESFEM discretization of  coupling a {\it regularized mean curvature flow} and diffusion on the surface:
$$
v - \alpha \laplace_{\Ga(t)} v= \Bigl( - H + g\bigl(u, \nabla_{\Ga(t)}u\bigr)\Bigr) \nu_{\Ga(t)},
$$
where $H$ denotes mean curvature on the surface $\Ga(t)$.

The error analysis is further extended to a {\it dynamic velocity law}
$$
 \mat v + v \nb_{\Ga(t)} \cdot v  - \alpha \laplace_{\Ga(t)} v= g\bigl(u, \nabla_{\Ga(t)}u\bigr)\, \nu_{\Ga(t)}.
$$
\bcl A physically more relevant dynamic velocity law would be based on momentum and mass balance, such as
incompressible Navier--Stokes motion of the surface coupled to diffusion on the surface. \ecl
We expect that our analysis extends to such a system, but this is beyond the scope of this paper. Surface evolutions under Navier--Stokes equations and under Willmore flow have recently been considered in \cite{BGN01,BGN02,BGN03}.

%
The paper is organized as follows.

In Section~\ref{section: problem} we describe the considered problems and give the weak formulation. We recall the basics of the evolving surface finite element method and describe the semidiscrete problem. Its matrix-vector formulation is useful not only for the implementation, but will play a key role in the stability analysis of this paper.

In Section~\ref{section: main result} we present the main result of the paper, which gives convergence estimates for the ESFEM semidiscretization with finite elements of polynomial degree at least $2$. We further outline the main ideas and the organization of the proof.

In Section~\ref{section: preliminaries} we present auxiliary results that are used to relate different surfaces to one another. They are the key technical results used later on in the stability analysis.
Section~\ref{section: uncoupled problem} contains the stability analysis for the regularized velocity law with a prescribed driving term.
In Section~\ref{section: coupled problem} this is extended to the stability analysis for coupling surface PDEs and surface motion. \bcl The stability analysis works with the matrix-vector formulation of the ESFEM semidiscretization and does not use geometric arguments.  \ecl

In Section~\ref{section: geometric estimates} we briefly recall some geometric estimates used for estimating the consistency errors, which are the defects obtained on inserting the interpolated exact solution into the scheme.
Section~\ref{section: defect bounds} deals with the defect estimates.
Section~\ref{section: proof completed} proves the main result by combining the results of the previous sections.

In Section~\ref{section: extensions} we give extensions to other velocity laws: the regularized mean curvature flow and the dynamic velocity law addressed above.

Section~\ref{section: numerics} presents numerical experiments that are complementary to our theoretical results in that they show the numerical behaviour of piecewise linear finite elements on some examples.

We use the notational convention to denote vectors in $\R^3$ by italic letters, but to denote finite element nodal vectors in $\R^N$ and $\R^{3N}$ by boldface lowercase letters and finite element mass and stiffness matrices by boldface capitals. All boldface symbols in this paper will thus be related to the matrix-vector formulation of the ESFEM.

\section{Problem formulation and evolving surface finite element semi\-discret\-ization}
\label{section: problem}

\subsection{Basic notions and notation}
\label{subsection: basic notions}
We consider the evolving two-dimensional closed surface $\Gamma(t)\subset\R^3$ as the image
$$
\Ga(t) = \{ X(p,t) \,:\, p \in \Ga^0 \}
$$
of a sufficiently regular vector-valued function $X:\Ga^0\times [0,T]\to \R^3$, where $\Ga^0$ is the smooth closed initial surface, and $X(p,0)=p$.
In view of the subsequent numerical discretization, it is convenient to think of $X(p,t)$ as the position at time $t$ of a moving particle with label $p$, and of $\Ga(t) $ as a collection of such particles.
To indicate the dependence of the surface on $X$, we will write
$$
\Ga(t) = \Ga(X(\cdot,t)), \quad\hbox{ or briefly}\quad \Ga(X)
$$
when the time $t$ is clear from the context. The {\it velocity} $v(x,t)\in\R^3$ at a point $x=X(p,t)\in\Gamma(t)$  equals
\begin{equation} \label{velocity}
 \partial_t X(p,t)= v(X(p,t),t).
\end{equation}
Note that for a known velocity field \bcl $v:\R^3\times [0,T] \to \R^3$, \ecl the position $X(p,t)$ at time $t$ of the particle with label $p$
is obtained by solving the ordinary differential equation \eqref{velocity} from $0$ to $t$ for a fixed $p$.

For a function $u(x,t)$ ($x\in \Gamma(t)$, $0\le t \le T$) we denote the {\it material derivative} as
$$
\mat u(x,t) = \frac \d{\d t} \,u(X(p,t),t) \quad\hbox{ for } \ x=X(p,t).
$$
At $x\in\Gamma(t)$ and $0\le t \le T$,
we denote by $\nu_{\Ga(X)}(x,t)$ the outer normal, by $\nabla_{\Ga(X)}u(x,t)$ the tangential gradient of $u$, by $\laplace_{\Ga(X)} u(x,t)$ the Laplace--Beltrami operator applied to $u$, and by $\nb_{\Ga(X)} \cdot v(x,t)$ the tangential divergence of $v$; see, e.g., \cite{DziukElliott_acta} for these notions.

\subsection{Surface motion coupled to a surface PDE: strong and weak formulation}
As outlined  in the introduction, we consider a parabolic equation on an evolving surface that moves according to an elliptically regularized velocity law:
\begin{equation}
\label{eq: coupled problem}
    \begin{aligned}
        \mat u + u \nb_{\Ga(X)} \cdot v - \laplace_{\Ga(X)} u   =&\  f(u, \nabla_{\Ga(X)} u), \\
        v - \alpha \laplace_{\Ga(X)} v =&\ g(u, \nb_{\Ga(X)} u ) \nu_{\Ga(X)}.
    \end{aligned}
\end{equation}
Here, $f:\R\times\R^3\to \R$ and $g:\R\times\R^3\to\R$ are given continuously differentiable functions,
and $\alpha>0$ is a fixed parameter. This system is considered together with the collection of ordinary differential equations \eqref{velocity} for every label $p$. Initial values are specified for $u$ and $X$.

On applying the Leibniz formula as in \cite{DziukElliott_ESFEM}, the weak formulation reads as follows:
\bcl Find  $u(\cdot,t)\in  W^{1,\infty}(\Ga(X(\cdot,t)))$ and $v(\cdot,t) \in  W^{1,\infty}(\Ga(X(\cdot,t)))^3 $  \ecl such that for all test functions $\vphi(\cdot,t)  \in H^1(\Ga(X(\cdot,t) ))$ with $\mat \vphi = 0$ and  $\psi(\cdot,t)  \in H^1(\Ga(X(\cdot,t) ))^3$,
\begin{equation}
\label{eq: coupled problem - weak form}
    \begin{aligned}
        \diff \int_{\Ga(X)} \!\! u \vphi + \int_{\Ga(X)} \!\! \nabla_{\Gamma(X)} u \cdot \nabla_{\Gamma(X)} \vphi =& \int_{\Ga(X)} f(u, \nabla_{\Gamma(X)}u) \vphi,
         \\[2mm] 
        \int_{\Ga(X)} v \cdot \psi + \alpha \int_{\Ga(X)} \!\! \nabla_{\Gamma(X)} v \cdot \nabla_{\Gamma(X)} \psi =&\ \int_{\Ga(X)}  g(u,\nabla_{\Gamma(X)} u) \,\nu_{\Ga(X)} \cdot \psi ,
    \end{aligned}
\end{equation}
alongside with the ordinary differential equations \eqref{velocity} for the positions $X$ determining the surface $\Gamma(X)$. 

\bcl We assume throughout this paper that the problem \eqref{eq: coupled problem} or \eqref{eq: coupled problem - weak form} admits a  unique solution with sufficiently high Sobolev regularity on the time interval $[0,T]$ for the given initial data $u(\cdot,0)$ and $X(\cdot,0)$. We assume further that the flow map $X(\cdot,t):\Gamma_0\to \Gamma(t)\subset\R^3$ is non-degenerate for $0\le t \le T$, so that $\Gamma(t)$ is a regular surface.\ecl

\subsection{Evolving surface finite elements}
\label{section:ESFEM}
We describe the surface finite element discretization of our problem, following \cite{Dziuk88} and \cite{Demlow2009}. We use simplicial elements and continuous piecewise polynomial basis functions of degree $k$, as defined in  \cite[Section 2.5]{Demlow2009}.

We triangulate the given smooth surface $\Gamma^0$ by an admissible family of triangulations $\mathcal{T}_h$ of decreasing maximal element diameter $h$; see \cite{DziukElliott_ESFEM} for the notion of an admissible triangulation, which includes quasi-uniformity and shape regularity. For a momentarily fixed $h$, we denote by $\bfx^0=(x_1^0,\dots,x_N^0)$ the vector in $\R^{3N}$ that collects all $N$ nodes of the triangulation. By piecewise polynomial interpolation of degree $k$, the nodal vector defines an approximate surface $\Gamma_h^0$ that interpolates $\Gamma^0$ in the nodes $x_j^0$. We will evolve the $j$th node in time, denoted $x_j(t)$ with $x_j(0)=x_j^0$, and collect the nodes at time $t$ in a vector
$$
\bfx(t) = (x_1(t),\dots,x_N(t)) \in \R^{3N}.
$$
Provided that $x_j(t)$ is sufficiently close to the exact position $x_j^*(t):=X(p_j,t)$ (with $p_j=x_j^0$) on the exact surface $\Gamma(t)=\Gamma(X(\cdot,t))$, the nodal vector $\bfx(t)$ still corresponds to an admissible triangulation. In the following discussion we omit the omnipresent argument $t$ and just write $\bfx$ for $\bfx(t)$ when the dependence on $t$ is not important.

By piecewise polynomial interpolation on the  plane reference triangle that corresponds to every
 curved triangle of the triangulation, the nodal vector $\bfx$ defines a closed surface denoted by $\Gamma_h(\bfx)$. We can then define finite element {\it basis functions}
$$
\phi_j[\bfx]:\Gamma_h(\bfx)\to\R, \qquad j=1,\dotsc,N,
$$
which have the property that on every triangle their pullback to the reference triangle is polynomial of degree $k$, and which satisfy
\begin{equation*}
    \phi_j[\bfx](x_k) = \delta_{jk} \quad  \text{ for all } j,k = 1,  \dotsc, N .
\end{equation*}
These functions span the finite element space on $\Gamma_h(\bfx)$,
\begin{equation*}
    S_h(\bfx) = \spn\big\{ \phi_1[\bfx], \phi_2[\bfx], \dotsc, \phi_N[\bfx] \big\} .
\end{equation*}
For a finite element function $u_h\in S_h(\bfx)$ the tangential gradient $\nabla_{\Gamma_h(\bfx)}u_h$ is defined piecewise.

We set
$$
X_h(p_h,t) = \sum_{j=1}^N x_j(t) \, \phi_j[\bfx(0)](p_h), \qquad p_h \in \Gamma_h^0,
$$
which has the properties that $X_h(p_j,t)=x_j(t)$ for $j=1,\dots,N$, that $X_h(p_h,0)=p_h$ for all $p_h\in\Gamma_h^0$, and
$$
\Gamma_h(\bfx(t))=\Gamma(X_h(\cdot,t)).
$$
The {\it discrete velocity} $v_h(x,t)\in\R^3$ at a point $x=X_h(p_h,t)\in \Gamma(X_h(\cdot,t))$ is given by
$$
\partial_t X_h(p_h,t) = v_h(X_h(p_h,t),t).
$$
A key property of the basis functions is the {\it transport property} \cite{DziukElliott_ESFEM}:
$$
\frac\d{\d t} \Bigl( \phi_j[\bfx(t)](X_h(p_h,t)) \Bigr) =0,
$$
which by integration from $0$ to $t$ yields
$$
\phi_j[\bfx(t)](X_h(p_h,t)) = \phi_j[\bfx(0)](p_h).
$$
This implies that the discrete velocity is simply
$$
v_h(x,t) = \sum_{j=1}^N v_j(t) \, \phi_j[\bfx(t)](x) \quad\hbox{ for } x \in \Gamma_h\bigl(\bfx(t)\bigr),
\qquad \hbox{with } \ v_j(t)=\dot x_j(t),
$$
where the dot denotes the time derivative $\d/\d t$.

The {\it discrete material derivative} of a finite element function
$$
u_h(x,t) = \sum_{j=1}^N u_j(t) \, \phi_j[\bfx(t)](x), \qquad x \in \Gamma_h\bigl(\bfx(t)\bigr),
$$
is defined as
$$
\mat_h u_h(x,t) = \frac\d{\d t} u_h(X_h(p_h,t),t) \qquad\hbox{for }\ x=X_h(p_h,t).
$$
By the transport property of the basis functions, this is just
$$
\mat_h u_h(x,t) = \sum_{j=1}^N \dot u_j(t) \, \phi_j[\bfx(t)](x), \qquad x \in \Gamma_h\bigl(\bfx(t)\bigr).
$$

\subsection{Semidiscretization of the evolving surface problem}
\label{subsection:semidiscretization}

The finite element spatial semidiscretization of the problem \eqref{eq: coupled problem - weak form} reads as follows: Find the unknown nodal vector $\bfx(t)\in \R^{3N}$ and the unknown finite element functions $u_h(\cdot,t)\in S_h(\bfx(t))$ and $v_h(\cdot,t)\in S_h(\bfx(t))^3$  such that, for all $\vphi_h(\cdot,t) \in S_h(\bfx(t))$ with $\mat_{h} \vphi_h=0$ and all $\psi_h(\cdot,t)\in S_h(\bfx(t))^3$,
\begin{equation}\label{uh-vh-equation}
    \begin{aligned}
        \diff \int_{\Ga_h(\bfx)} \!\!\! u_h \vphi_h + \int_{\Ga_h(\bfx)}\!\!\!\!\! \nabla_{\Ga_h(\bfx)} u_h \cdot \nabla_{\Ga_h(\bfx)}  \vphi_h
        =& \ \int_{\Ga_h(\bfx)} \!\!\! f(u_h,\nabla_{\Ga_h(\bfx)}  u_h)\,\varphi_h, \\
        \int_{\Ga_h(\bfx)}\!\!\! v_h \cdot \psi_h + \alpha \int_{\Ga_h(\bfx)} \!\!\!\!\! \nabla_{\Ga_h(\bfx)}  v_h \cdot \nabla_{\Ga_h(\bfx)}  \psi_h =&\ \int_{\Ga_h(\bfx)} \!\!\! g(u_h,\nabla_{\Ga_h(\bfx)}  u_h) \,\nu_{\Ga_h(\bfx)} \cdot \psi_h,
    \end{aligned}
\end{equation}
and
\begin{equation}\label{xh}
\partial_t X_h(p_h,t) = v_h(X_h(p_h,t),t), \qquad p_h\in\Ga_h^0.
\end{equation}
The initial values for the nodal vector $\bfu$ corresponding to $u_h$ and the nodal vector $\bfx$ of the initial positions are taken as the exact initial values at the nodes $x_j^0$ of the triangulation of the given initial surface $\Gamma^0$:
$$
x_j(0) = x_j^0, \quad u_j(0)=u(x_j^0,0) , \qquad (j=1,\dotsc,N).
$$

\subsection{Differential-algebraic equations of the matrix-vector formulation}\label{subsection:DAE}
We now show that the nodal vectors $\bfu\in\R^N$ and $\bfv\in\R^{3N}$ of the finite element functions $u_h$ and $v_h$, respectively,  together with the surface nodal vector $\bfx\in\R^{3N}$ satisfy a system of differential-algebraic equations (DAEs).
Using the above finite element setting,  we set (omitting the argument $t$)
\begin{align*}
    u_h =&\ \sum_{j=1}^N u_j \phi_j[\bfx], \qquad u_h(x_j)=u_j \in \R , \\
    v_h =&\ \sum_{j=1}^N v_j \phi_j[\bfx], \qquad v_h(x_j)=v_j \in \R^3 ,
\end{align*}
and collect the nodal values in column vectors $\bfu=(u_j)\in\R^N$ and $\bfv=(v_j) \in \R^{3N}$.

We define the surface-dependent mass matrix $\bfM(\bfx)$ and stiffness matrix $\bfA(\bfx)$ on the surface determined by the nodal vector $\bfx$:
\begin{equation*}
    \begin{aligned}
        \bfM(\bfx)\vert_{jk} =&\ \int_{\Ga_h(\bfx)} \! \phi_j[\bfx] \phi_k[\bfx] , \\
        \bfA(\bfx)\vert_{jk} =&\ \int_{\Ga_h(\bfx)} \! \nb_{\Ga_h} \phi_j[\bfx] \cdot \nb_{\Ga_h} \phi_k[\bfx] ,
    \end{aligned}
    \qquad (j,k = 1,  \dotsc, N) .
\end{equation*}
We further let \bbk (with the identity matrix $I_3 \in \R^{3 \times 3}$) \ebk
\begin{equation}
\label{eq: K matrix def}
    \bfK(\bfx)= I_3 \otimes \Bigl(\bfM(\bfx) + \alpha \bfA(\bfx)\Bigr).  
\end{equation}
The right-hand side vectors  $\bff(\bfx,\bfu)\in\R^N$ and $\bfg(\bfx,\bfu)\in\R^{3N}$ are given by
\begin{equation*}
 \begin{aligned}
    \bff(\bfx,\bfu)\vert_j &= \int_{\Ga_h(\bfx)} f(u_h,\nbgh u_h) \, \phi_j[\bfx],
    \\
    \bfg(\bfx,\bfu)\vert_{3(j-1)+\ell} &= \int_{\Ga_h(\bfx)} g(u_h,\nbgh u_h) \,\bigl(\nu_{\Ga_h(\bfx)}\bigr)_\ell \, \phi_j[\bfx],
    \end{aligned}
\end{equation*}
for $j = 1,  \dotsc, N,$ and $\ell=1,2,3$.

We then obtain from \eqref{uh-vh-equation}--\eqref{xh} the following coupled DAE system for the nodal values $\bfu, \bfv$ and $\bfx$:
\begin{equation}
\label{eq: DAE form - coupled}
    \begin{aligned}
        \diff \Big(\bfM(\bfx)\bfu\Big) + \bfA(\bfx)\bfu =&\   \bff(\bfx,\bfu), \\
        \bfK(\bfx) \bfv =&\ \bfg(\bfx,\bfu),\\
        \dot \bfx =&\ \bfv.
    \end{aligned}
\end{equation}
With the auxiliary vector $\bfw=\bfM(\bfx)\bfu$, this system becomes
\begin{equation*}
    \begin{aligned}
     \dot \bfx =&\ \bfv, \\
  \dot \bfw  =&\  - \bfA(\bfx)\bfu +\bff(\bfx,\bfu), \\
     \bfzero = &\  -\bfK(\bfx) \bfv + \bfg(\bfx,\bfu),\\
     \bfzero = &\ - \bfM(\bfx)\bfu + \bfw.
    \end{aligned}
\end{equation*}
This is of a form to which standard DAE time discretization can be applied; see, e.g., \cite[Chap.\,VI]{HairerWannerII}.

\bcl
As will be seen in later sections, the matrix-vector formulation is very useful in the stability analysis of the ESFEM, beyond its obvious role for practical computations.
\ecl

\subsection{Lifts} \label{section:lifts}
In the error analysis we need to compare functions on three different surfaces: the {\it exact surface} $\Gamma(t)=\Gamma(X(\cdot,t))$, the {\it discrete surface} $\Gamma_h(t)=\Gamma_h(\bfx(t))$, and the {\it interpolated surface} $\Gamma_h^*(t)=\Gamma_h(\bfx^*(t))$, where $\bfx^*(t)$ is the nodal vector collecting the grid points $x_j^*(t)=X(p_j,t)$ on the exact surface. In the following definitions we omit the argument $t$ in the notation.

A finite element function $w_h:\Gamma_h\to\R^m$ ($m=1$ or 3) on the discrete surface, with nodal values $w_j$, is related to the finite element function $\widehat w_h$ on the interpolated surface that has the same nodal values:
$$
\widehat w_h  = \sum_{j=1}^N w_j \phi_j[\xs].
$$
 The transition between the interpolated surface and the exact surface is  done by the
 \emph{lift operator}, which was introduced \bbk for linear surface approximations \ebk in \cite{Dziuk88}; \bbk see also \cite{DziukElliott_ESFEM,DziukElliott_L2}.
 Higher-order generalizations have been studied in \cite{Demlow2009}. \ebk The lift operator $l$ maps a function on the interpolated surface $\Gamma_h^*$ to a function on the exact surface $\Gamma$, provided that $\Gamma_h^*$ is sufficiently close to $\Gamma$.

\bbk
The exact regular surface $\Ga(X(\cdot,t))$ can be represented by a (sufficiently smooth) signed distance function $d : \R^3 \times [0,T] \to \R$, cf.~\cite[Section~2.1]{DziukElliott_ESFEM}, such that
\begin{equation}
\label{eq: distance function}
	\Ga(X(\cdot,t)) = \big\{ x\in \R^3 \mid d(x,t) = 0 \big\} \subset \R^3 .
\end{equation}

Using this distance function, \ebk the lift of a continuous function $\eta_h \colon \Ga_h^* \to \R$ is defined as
\begin{equation*}
    \eta_{h}^{l}(y) := \eta_h(x), \qquad x\in\Ga_h^*,
\end{equation*}
where for every $x\in \Ga_h^*$ the point $y=y(x)\in\Ga$ is uniquely defined via
\begin{equation*}
\label{eq: lift defining equation}
    y = x - \nu(y) d(x).
\end{equation*}
For functions taking values in $\R^3$ the lift is componentwise.
By $\eta^{-l}$ we denote the function on $\Ga_h^*$ whose lift is $\eta$.

We denote the composed lift $L$ from finite element functions on $\Gamma_h$ to functions on $\Gamma$ via $\Gamma_h^*$ by
$$
w_h^L = (\widehat w_h)^l.
$$

%
%

%
\section{Statement of the main result: semidiscrete error bound}
\label{section: main result}


We are now in the position to formulate the main result of this paper, which yields optimal-order error bounds for the finite element semidiscretization of  a surface PDE on a solution-driven surface as specified in \eqref{eq: coupled problem}, for finite elements of polynomial degree $k\ge 2$. We denote by $\Gamma(t)=\Gamma(X(\cdot,t))$ the exact surface and by $\Gamma_h(t)=\Gamma(X_h(\cdot,t))=\Gamma_h(\bfx(t))$ the discrete surface at time $t$. We introduce the notation
$$
x_h^L(x,t) =  X_h^L(p,t) \in \Gamma_h(t) \qquad\hbox{for}\quad x=X(p,t)\in\Gamma(t).
$$

\begin{theorem}
\label{theorem: coupled error estimate}
    Consider the space discretization \eqref{uh-vh-equation}--\eqref{xh} of the coupled problem  \eqref{velocity}--\eqref{eq: coupled problem}, using evolving surface finite elements of polynomial degree~$k\ge 2$.
    We assume quasi-uniform admissible triangulations of the initial surface and initial values chosen by finite element interpolation of the initial data for $u$.
    Suppose that the problem admits an exact solution $(u, v,X)$ that is sufficiently smooth (say, in the Sobolev class $H^{k+1}$) on the time interval $0\le t \le T$, \bcl and that the flow map $X(\cdot,t):\Gamma_0\to \Gamma(t)\subset\R^3$ is non-degenerate for $0\le t \le T$, so that $\Gamma(t)$ is a regular surface. \ecl

    Then, there exists $h_0 >0$ such that for all mesh widths $h \leq h_0$ the following error bounds hold over the exact surface $\Ga(t)=\Ga(X(\cdot,t))$
     for $0\le t \le T$:
    \begin{equation*}
            \bigg(\|u_h^L(\cdot,t) - u(\cdot,t)\|_{L^2(\Ga(t))}^2 + \int_0^t \|u_h^L(\cdot,s) - u(\cdot,s)\|_{H^1(\Ga(s))}^2 \,\d s \\
            \bigg)^\half \leq C h^k
    \end{equation*}
    and
    \begin{equation*}
     \begin{aligned}
  \bcl  \biggl(  \int_0^t   \|v_h^{L}(\cdot,s) - v(\cdot,s)\|_{H^1(\Ga(s))^3} ^2\,\d s\biggr)^{1/2} \ecl&\leq Ch^k, \\
       \|x_h^L(\cdot,t) - \mathrm{id}_{\Gamma(t)}\|_{H^1(\Ga(t))^3} &\leq Ch^k.
     \end{aligned}
    \end{equation*}
    The constant $C$ is independent of $t$ and $h$, but depends on bounds of the $H^{k+1}$ norms of the solution $(u,v,X)$, on local Lipschitz constants of $f$ and $g$,
on the regularization parameter $\alpha>0$ and on the length $T$ of the time interval.
\end{theorem}

We note that the last error bound is equivalent to
$$
 \|X_h^L(\cdot,t) - X(\cdot,t)\|_{H^1(\Ga^0)^3} \leq ch^k.
$$
\bcl Moreover, in the case of a coupling function $g$ in \eqref{eq: coupled problem} that is independent of the solution gradient, so that $g=g(u)$, we obtain an error bound for the velocity that is pointwise in time: uniformly for $0\le t \le T$,
$$
 \|v_h^{L}(\cdot,t) - v(\cdot,t)\|_{H^1(\Ga(t))^3} \le Ch^k.
 $$
 \ecl

A key issue in the proof is to ensure that the $W^{1,\infty}$ norm of the position error of the curves remains small. The $H^1$ error bound and an inverse estimate yield an $O(h^{k-1})$ error bound in the $W^{1,\infty}$ norm. This is small only for $k\ge 2$, which is why we impose the condition $k\ge 2$ in the above result.

\bcl Since the exact flow map $X(\cdot,t):\Gamma_0\to\Gamma(t)$ is assumed to be smooth and non-degenerate, it is locally close to an invertible linear  transformation, and (using compactness) it therefore preserves the admissibility of grids with sufficiently small mesh width $h\le h_0$. \ecl
Our assumptions therefore guarantee that the triangulations formed by the nodes $x_j^*(t)=X(p_j,t)$ remain admissible uniformly for $t\in[0,T]$ \bcl for sufficiently small $h$ (though the bounds in the admissibility inequalities and the largest possible mesh width may deteriorate with growing time). \ecl Since $k\ge 2$, the position error estimate implies that for sufficiently small $h$ also the triangulations formed by the numerical nodes $x_j(t)$ remain admissible uniformly for $t\in[0,T]$. This cannot be concluded for $k=1$.

The error bound will be proven by clearly separating the issues of consistency and stability.
The
consistency error is the  defect on inserting a projection (interpolation or Ritz projection) of the exact solution into the discretized equation.
The defect bounds involve {\it geometric estimates} that were obtained for \bbk the time dependent case and for \ebk higher order $k\ge 2$ in \cite{highorder}, by combining techniques of Dziuk \& Elliott \cite{DziukElliott_ESFEM,DziukElliott_L2} and Demlow \cite{Demlow2009}. This is done with the  ESFEM formulation of Section~\ref{subsection:semidiscretization}.

The main issue in the proof of Theorem~\ref{theorem: coupled error estimate} is to prove {\it stability} in the form of an $h$-independent bound of the error in terms of the defect.  The stability analysis is  done in the matrix-vector formulation of Section~\ref{subsection:DAE}.  It uses energy estimates and transport formulae that relate the mass and stiffness matrices and the coupling terms for different  nodal vectors $\bfx$.
 {\it No geometric estimates} enter in the proof of stability. 

 In Section~\ref{section: preliminaries} we prove important auxiliary results for the stability analysis. The stability is first analysed for the discretized velocity law without coupling to the surface PDE in Section~\ref{section: uncoupled problem} and is then extended to the coupled problem in Section~\ref{section: coupled problem}. The necessary geometric estimates for the consistency analysis are collected in Section~\ref{section: geometric estimates}, and the defects are then bounded in Section~\ref{section: defect bounds}. The proof of Theorem~\ref{theorem: coupled error estimate} is then completed in Section~\ref{section: proof completed} by putting together the results on stability, defect bounds and interpolation error bounds.

%

\section{Auxiliary results for the stability analysis: relating different surfaces}
\label{section: preliminaries}

The  finite element matrices of Section~\ref{subsection:DAE} induce discrete versions of Sobolev norms. For any $\bfw=(w_j) \in \R^N$ with corresponding finite element function $w_h= \sum_{j=1}^N w_j \phi_j[\bfx] \in S_h(\bfx)$ we note
\begin{align} \label{M-L2}
   &  \|\bfw\|_{\bfM(\bfx)}^{2} := \bfw^T \bfM(\bfx) \bfw = \|w_h\|_{L^2(\Ga_h(\bfx))}^2, \\
   \label{A-H1}
   &  \|\bfw\|_{\bfA(\bfx)}^{2} := \bfw^T \bfA(\bfx) \bfw = \|\nb_{\Ga_h(\bfx)} w_h\|_{L^2(\Ga_h(\bfx))}^2.
\end{align}
In our stability analysis we need to relate  finite element matrices corresponding to different nodal vectors.
%
%
%
%
%
We
use the following setting.
Let $\bfx,\bfy \in \R^{3 N}$ be two nodal vectors defining discrete surfaces $\Gamma_h(\bfx)$ and $\Gamma_h(\bfy)$, respectively. We let $\bfe= (e_j)=\bfx-\bfy \in \R^{  3  N}$. For the parameter $\theta\in[0,1]$, we consider the intermediate surface $\Gamma_h^\theta=\Gamma_h(\bfy+\theta\bfe)$ and the corresponding finite element functions given as
$$
e_h^\theta=\sum_{j=1}^N e_j \phi_j[\bfy+\theta\bfe]
$$
and, for any vectors $\bfw,\bfz \in \R^N$,
$$
    w_h^\theta=\sum_{j=1}^N w_j \phi_j[\bfy+\theta\bfe] \andquad z_h^\theta=\sum_{j=1}^N z_j \phi_j[\bfy+\theta\bfe] .
$$

\begin{lemma}
\label{lemma: matrix differences}
    In the above setting the following identities hold:
    \begin{align*}
        \bfw^T (\bfM(\bfx)-\bfM(\bfy)) \bfz =&\ \int_0^1 \int_{\Ga_h^\theta} w_h^\theta (\nabla_{\Ga_h^\theta} \cdot e_h^\theta) z_h^\theta \; \d\theta, \\
        \bfw^T (\bfA(\bfx)-\bfA(\bfy)) \bfz =&\ \int_0^1 \int_{\Ga_h^\theta} \nb_{\Ga_h^\theta} w_h^\theta \cdot (D_{\Ga_h^\theta} e_h^\theta)\nb_{\Ga_h^\theta}  z_h^\theta \; \d\theta ,
    \end{align*}
    with
    $D_{\Ga_h^\theta} e_h^\theta =  \textnormal{trace}(E) I_3 - (E+E^T)$ for $E=\nabla_{\Ga_h^\theta} e_h^\theta \in \R^{3\times 3}$.
\end{lemma}
\begin{proof}
Using the fundamental theorem of calculus and the Leibniz formula we write
\begin{align*}
    \bfw^T (\bfM
(\bfx)-\bfM(\bfy)) \bfz =&\ \int_{\Gamma_h(\bfx)} \!\!\! w_h^1 z_h^1 - \int_{\Gamma_h(\bfy)} \!\!\! w_h^0 z_h^0
    = \int_0^1 \frac{\d}{\d \theta} \int_{\Gamma_h^\theta} w_h^\theta z_h^\theta \d\theta \\
    =&\ \int_0^1 \int_{\Gamma_h^\theta} w_h^\theta (\nabla_{\Gamma_h^\theta} \cdot e_h^\theta) z_h^\theta \; \d\theta .
\end{align*}
In the last formula we used that the material derivatives (with respect to $\theta$) of $w_h^\theta$ and $z_h^\theta$ vanish, thanks to the
transport property of the basis functions.
The second identity is shown in the same way, using the formula for the derivative of the Dirichlet integral; see \cite{DziukElliott_ESFEM} and also \cite[Lemma 3.1]{DziukLubichMansour_rksurf}.
\QED \end{proof}

A direct consequence of Lemma~\ref{lemma: matrix differences} is the following conditional equivalence of norms:

\begin{lemma}
\label{lemma:cond-equiv}
If $\| \nabla_{\Gamma_h^\theta} \cdot e_h^\theta \|_{L^\infty(\Gamma_h^\theta)} \le \mu$ for $0\le\theta\le 1$, then
$$
 \| \bfw \|_{\bfM(\bfy+\bfe)}  \le e^{\mu/2} \, \| \bfw \|_{\bfM(\bfy)}.
$$
If $\| D_{\Gamma_h^\theta} e_h^\theta \|_{L^\infty(\Gamma_h^\theta)} \le \eta$ for $0\le\theta\le 1$, then
$$
 \| \bfw \|_{\bfA(\bfy+\bfe)}  \le e^{\eta/2} \, \| \bfw \|_{\bfA(\bfy)}.
$$
\end{lemma}

\begin{proof} By Lemma~\ref{lemma: matrix differences} we have for $0\le \tau \le 1$
\begin{align*}
   & \|\bfw\|_{\bfM(\bfy+\tau\bfe)}^2 - \|\bfw\|_{\bfM(\bfy)}^2 =\ \bfw^T ( \bfM(\bfy+\tau\bfe)-\bfM(\bfy))\bfw \\
   & =\ \int_0^\tau \int_{\Gamma_h^\theta} w_h^\theta\cdot (\nabla_{\Gamma_h^\theta} \cdot e_h^\theta) w_h^\theta \d\theta
    \leq \  \mu \int_0^\tau \| w_h^\theta \|_{L^2(\Gamma_h^\theta)}^2 \, \d\theta \\
    &=\ \mu \int_0^\tau \|\bfw\|_{\bfM(\bfy+\theta\bfe)}^2 \, \d\theta,
\end{align*}
and the first result follows from Gronwall's inequality. The second result is proved in the same way.
\QED \end{proof}

\bcl The following result, when used with $w_h^\theta$ equal to components of $e_h^\theta$, reduces the problem of checking
the conditions of the previous lemma  for  $0\le\theta\le 1$ to checking the condition just for the case $\theta=0$. \ecl

\begin{lemma} \label{lemma:theta-independence}
In the above setting, assume that
\begin{equation}\label{e-inf-bound}
\| \nabla_{\Gamma_h[\bfy]} e_h^0 \|_{L^\infty(\Gamma_h[\bfy])} \le \frac12.
\end{equation}Then, for $0\le\theta\le 1$ the function $w_h^\theta=\sum_{j=1}^N w_j \phi_j[\bfy+\theta\bfe]$ on $\Gamma_h^\theta=\Gamma[\bfy+\theta\bfe]$ is bounded by
$$
\| \nabla_{\Gamma_h^\theta} w_h^\theta \|_{L^p(\Gamma_h^\theta)} \le c_p \, \| \nabla_{\Gamma_h^0} w_h^0 \|_{L^p(\Gamma_h^0)}
\for 1\le p \le \infty,
$$
where $c_p$ depends only on $p$ (we have $c_\infty=2$).
\end{lemma}

\begin{proof} We describe the finite element parametrization of the discrete surfaces $\Gamma_h^\theta$ in the same way as in Section~\ref{section:ESFEM}, with $\theta$ instead of $t$ in the role of the time variable. We set
\begin{equation}\label{Y-h-theta}
Y_h^\theta(q_h) = Y_h(q_h,\theta) = \sum_{j=1}^N (y_j + \theta e_j) \phi_j[\bfy](q_h),\quad\ q_h\in\Gamma_h[\bfy] ,
\end{equation}
so that
$$
 \Gamma(Y_h^\theta)=\Gamma_h[\bfy+\theta\bfe]=\Gamma_h^\theta.
$$
Since $Y_h^0(q_h)=q_h$ for all $q_h\in\Gamma_h^0=\Gamma_h[\bfy]$,
the above formula can be rewritten as
$$
Y_h^\theta(q_h) = q_h + \theta e_h^0(q_h).
$$
Tangent vectors to $\Gamma_h^\theta$ at $y_h^\theta=Y_h^\theta(q_h)$ are therefore of the form
$$
\delta y_h^\theta = DY_h^\theta(q_h)\,\delta q_h = \delta q_h + \theta  \bigl(\nabla_{\Gamma_h^0} e_h^0(q_h)\bigr)^T \delta q_h,
$$
where $\delta q_h$ is a tangent vector to $\Gamma_h^0$ at $q_h$, or written more concisely,
$\delta q_h \in T_{q_h} \Gamma_h^0$.

Letting $|\cdot|$ denote the Euclidean norm of a vector in $\R^3$, we have at
$y_h^\theta=Y_h^\theta(q_h)$
\begin{align*}
|  \nabla_{\Gamma_h^\theta} w_h^\theta (y_h^\theta) |
&= \sup_{\delta y_h^\theta \in T_{y_h^\theta} \Gamma_h^\theta}
\frac{\bigl( \nabla_{\Gamma_h^\theta} w_h^\theta (y_h^\theta)\bigr)^T \delta y_h^\theta }{|\delta y_h^\theta|}
= \sup_{\delta y_h^\theta \in T_{y_h^\theta} \Gamma_h^\theta}
\frac{D w_h^\theta (y_h^\theta) \delta y_h^\theta }
{|\delta y_h^\theta|}
\\
&=  \sup_{\delta q_h \in T_{q_h} \Gamma_h^0}
\frac{D w_h^\theta (y_h^\theta) DY_h^\theta(q_h)\,\delta q_h  }
{|DY_h^\theta(q_h)\,\delta q_h |} .
\end{align*}
By construction of $w_h^\theta$ and the transport property of the basis functions, we have
$$
w_h^\theta (Y_h^\theta(q_h))= \sum_{j=1}^N w_j \phi_j[\bfy+\theta\bfe](Y_h^\theta(q_h))
= \sum_{j=1}^N w_j \phi_j[\bfy](q_h) = w_h^0(q_h).
$$
By the chain rule, this yields
$$
 D w_h^\theta (y_h^\theta) DY_h^\theta(q_h) = D w_h^0(q_h) .
$$
Under the imposed condition $\| \nabla_{\Gamma_h^0} e_h^0 \|_{L^\infty(\Gamma_h[\bfy])} \le \frac12$ we have for $0\le\theta\le 1$
$$
|DY_h^\theta(q_h)\,\delta q_h | \ge |\delta q_h| - \theta |\bigl(\nabla_{\Gamma_h^0} e_h^0(q_h)\bigr)^T \delta q_h| \ge
\tfrac12  |\delta q_h|.
$$
Hence we obtain
\begin{align*}
|  \nabla_{\Gamma_h^\theta} w_h^\theta (y_h^\theta) |  &=
\sup_{\delta q_h \in T_{q_h} \Gamma_h^0}
\frac{D w_h^0(q_h)\,\delta q_h  }
{|DY_h^\theta(q_h)\,\delta q_h |}
\\
&\le
\sup_{\delta q_h \in T_{q_h} \Gamma_h^0}
\frac{D w_h^0(q_h)\,\delta q_h  }
{\tfrac12 |\delta q_h |} = 2\, |\nabla_{\Gamma_h^0}w_h^0(q_h)|.
\end{align*}
This yields the stated result for $p=\infty$. For $1\le p <\infty$ we note in addition that in using the integral transformation formula we have a uniform bound between the surface elements, since $DY_h^\theta$ is close to the identity matrix by our smallness assumption on $\nabla_{\Gamma_h^0} e_h^0$.
\QED \end{proof}

%
%
%
%
The arguments of the previous proof are also used in estimating the changes of the normal vectors on the various surfaces $\Gamma_h^\theta=\Gamma_h[\bfy+\theta \bfe]$.

\begin{lemma}
\label{lemma: normal vector perturbation}
Suppose that condition \eqref{e-inf-bound} is satisfied.
Let $y_h^\theta=Y_h^\theta(q_h)\in\Gamma_h^\theta$ be related by the parametrization \eqref{Y-h-theta} of $\Gamma_h^\theta$ over $\Gamma_h^0$, for $0\le\theta\le 1$. Then, the corresponding unit normal vectors differ by no more than
 $$
 |\nu_{\Gamma_h^\theta}(y_h^\theta) - \nu_{\Gamma_h^0}(y_h^0)| \le C\theta  | \nabla_{\Gamma_h^0} e_h^0(y_h^0) | ,
 $$
 with some constant $C$.
\end{lemma}

\begin{proof}
 Let $\delta q_h^1$ and $\delta q_h^2$ be two linearly independent tangent vectors of $\Gamma_h^0$ at $q_h\in\Gamma_h^0$
 (which may be chosen orthogonal to each other and \bbk of unit length with respect to the Euclidean norm\ebk). With
 $\delta y_h^{\theta,i} = DY_h^\theta(q_h)\,\delta q_h^i = \delta q_h^i + \theta  \bigl(\nabla_{\Gamma_h^0} e_h(q_h)\bigr)^T \delta q_h^i$ for $i=1,2$ we then have, for $0\le\theta\le 1$,
 $$
 \nu_{\Gamma_h^\theta}(y_h^\theta) = \frac{\delta y_h^{\theta,1}\times \delta y_h^{\theta,2}}{|\delta y_h^{\theta,1}\times \delta y_h^{\theta,2}|}.
 $$
 Since this expression is a locally Lipschitz continuous function of the two vectors,
the result follows. (The imposed bound \eqref{e-inf-bound} is sufficient to ensure the linear independence of the vectors $\delta y_h^{\theta,i}$.)
\QED \end{proof}

We denote by $\mat_\theta f$ the material derivative of a function $f=f(y_h^\theta,\theta)$ depending on $\theta\in[0,1]$ and $y_h^\theta\in\Gamma_h^\theta$:
$$
\mat_\theta f = \frac\d{\d\theta} f (y_h^\theta,\theta).
$$
From Lemma~\ref{lemma: normal vector perturbation} together with
Lemma~\ref{lemma:theta-independence} we  obtain the following bound:
\begin{lemma}\label{lemma: mat normal vector} If condition \eqref{e-inf-bound} is satisfied, then
$$
\| \mat_\theta  \nu_{\Gamma_h^\theta} \|_{L^p(\Gamma_h^\theta)} \le C
\| \nabla_{\Gamma_h^0} e_h^0 \|_{L^p(\Gamma_h^0)} ,
$$
where $C$ is independent of $0\le\theta\le 1$ and $1\le p \le \infty$.
\end{lemma}

\begin{proof} By Lemma~\ref{lemma: normal vector perturbation} with
$\Gamma_h^\theta$ in the role of $\Gamma_h^0$, we obtain
$$
|\mat_\theta  \nu_{\Gamma_h^\theta} (y_h^\theta)| = \bigl| \lim_{\tau\to 0}
\bigl( \nu_{\Gamma_h^{\theta+\tau}}(y_h^{\theta+\tau}) -  \nu_{\Gamma_h^\theta}(y_h^\theta) \bigr)/\tau \bigr|
\le C | \nabla_{\Gamma_h^\theta} e_h^\theta(y_h^\theta)|,
$$
which implies
$$
\| \mat_\theta  \nu_{\Gamma_h^\theta} \|_{L^p(\Gamma_h^\theta)} \le C
\| \nabla_{\Gamma_h^\theta} e_h^\theta \|_{L^p(\Gamma_h^\theta)},
$$
and Lemma~\ref{lemma:theta-independence} completes the proof.
\QED \end{proof}

We finally need a result that bounds the time derivatives of the mass and stiffness matrices corresponding to
nodes on the exact smooth surface $\Gamma(t)$.  The following result is a
direct consequence of \cite[Lemma 4.1]{DziukLubichMansour_rksurf} .

\begin{lemma}
\label{lemma: matrix derivatives}
    Let $\Gamma(t)=\Gamma(X(\cdot,t))$, $t \in [0,T]$, be a smoothly evolving family of smooth closed surfaces, and
    let the vector $\xs(t) \in \R^{3N}$ collect the nodes $x_j^*(t)=X(p_j,t)$.  Then,
    \begin{align*}
        \bfw^T \frac\d{\d t}\bfM(\xs(t))  \bfz \leq&\ C  \,\|\bfw\|_{\bfM(\xs(t))}\|\bfz\|_{\bfM(\xs(t))} , \\
        \bfw^T \frac\d{\d t}\bfA(\xs(t))  \bfz \leq&\ C \, \|\bfw\|_{\bfA(\xs(t))}\|\bfz\|_{\bfA(\xs(t))} ,
    \end{align*}
    for all $\bfw,\bfz\in \R^N$. 
    The constant $C$ depends only on a bound of the $W^{1,\infty}\!$ norm of the surface velocity.
\end{lemma}

\section{Stability of discretized surface motion under a prescribed driving-term}
\label{section: uncoupled problem}

\bcl In this  section we begin the stability analysis by first studying the stability of the spatially discretized velocity law with a given inhomogeneity instead of a coupling to the surface PDE. This allows us to present, in a technically simpler setting, some of the basic arguments that are used in our approach to stability estimates, which works with the matrix-vector formulation. The stability of the spatially discretized problem including coupling with the surface PDE is then studied in Section~\ref{section: coupled problem} by similar, but more elaborate arguments.
\ecl

\subsection{Uncoupled velocity law and its semidiscretization}

In this section we consider the velocity law without coupling to a surface PDE:
\begin{equation*}
    \begin{aligned}
        v - \alpha \laplace_{\Ga(X)} v =&\ g\, \nu_{\Ga(X)}, \\
    \end{aligned}
\end{equation*}
where $g:\R^3\times\R\to\R$ is a given continuous function of $(x,t)$, and $\alpha>0$ is a fixed parameter. This problem is considered together with the ordinary differential equations \eqref{velocity} for the positions $X$ determining the surface $\Gamma(X)$. Initial values are specified for~$X$.

\bcl The weak formulation is given by the second formula of \eqref{eq: coupled problem - weak form}
with the function $g$ considered here. This is considered together with the ordinary differential equations \eqref{velocity} for the positions $X$. \ecl
%

Then the finite element spatial semidiscretization of this problem reads as: Find the unknown nodal vector $\bfx(t)\in \R^{3N}$ and the unknown finite element function $v_h(\cdot,t)\in S_h(\bfx(t))^3$ such that the following semidiscrete equation holds for every $\psi_h \in S_h(\bfx(t))^3$:
\begin{equation}
\label{eq: uncoupled problem - semidiscrete}
    \begin{aligned}
        \int_{\Ga_h(\bfx)} \!\! v_h \cdot \psi_h + \alpha \int_{\Ga_h(\bfx)} \!\! \!\!   \nabla_{\Ga_h(\bfx)} v_h \cdot \nabla_{\Ga_h(\bfx)} \psi_h =&\ \int_{\Ga_h(\bfx)} \!\! g \,\nu_{\Ga_h(\bfx)} \cdot \psi_h ,
    \end{aligned}
\end{equation}
together with the ordinary differential equations \eqref{xh}. As before, the nodal vector of the initial positions $\bfx(0)$ is taken from the exact initial values at the nodes $x_j^0$ of the triangulation of the given initial surface $\Gamma^0$: $x_j(0) = x_j^0$ for $j=1,\dots,N$.

As in Section~\ref{subsection:DAE}, the nodal vectors $\bfv\in\R^{3N}$ of the finite element function $v_h$  together with the surface nodal vector $\bfx\in\R^{3N}$ satisfy a system of differential-algebraic equations (DAEs).
We obtain from \eqref{eq: uncoupled problem - semidiscrete} and \eqref{xh} the following coupled DAE system for the nodal values $\bfv$ and $\bfx$:
\begin{equation}
\label{eq: DAE form - uncoupled}
    \begin{aligned}
        \bfK(\bfx) \bfv =&\ \bfg(\bfx,t) , \\
        \dot\bfx =&\ \bfv .
    \end{aligned}
\end{equation}
Here the matrix $\bfK(\bfx)=I_3\otimes (\bfM(\bfx)+\alpha\bfA(\bfx))$ is from \eqref{eq: K matrix def}, and the driving term $\bfg(\bfx,t)$ is given by
\begin{equation*}
    \bfg(\bfx,t))\vert_{3(j-1)+\ell} = \int_{\Ga_h(\bfx)} g(\cdot,t) \,\bigl(\nu_{\Ga_h(\bfx)}\bigr)_\ell \, \phi_j[\bfx] , \qquad (j = 1,  \dotsc, N, \ \ell=1,2,3) .
\end{equation*}

\subsection{Error equations}

We denote by
$$
\bfx^*(t)=\bigl( x_j^*(t)\bigr)\in \R^{3N}\quad\hbox{with}\quad x_j^*(t)=X(p_j,t) , \qquad (j=1,\dots,N)
$$
the nodal vector of the {\it exact} positions on the surface $\Gamma(X(\cdot,t))$. This defines a discrete surface $\Gamma_h(\bfx^*(t))$ that interpolates the exact surface $\Gamma(X(\cdot,t))$.
%

We consider \bcl the interpolated  exact velocity
$$
 v_h^*(\cdot, t)=\sum_{j=1}^N v_j^*(t) \phi_j[\bfx^*(t)] \quad\ \hbox{ with }\quad\ v_j^*(t)=\dot x_j^*(t),
$$
with the corresponding nodal vector
$$
\bfv^*(t)=\bigl( v_j^*(t)\bigr)= \dot \bfx^*(t) \in \R^{3N}.
$$
\ecl
Inserting $v_h^*$ and $\bfx^*$ in place of the numerical solution $v_h$ and $\bfx$ into \eqref{eq: uncoupled problem - semidiscrete} yields a defect $d_h(\cdot,t)\in S_h(\bfx^*(t))^3$: for every $\psi_h \in S_h(\bfx^*(t))^3$,
$$
    \int_{\Ga_h(\bfx^*)} \!\!\!\! v_h^* \cdot \psi_h + \alpha \int_{\Ga_h(\bfx^*)} \!\!\!\!\!    \nabla_{\Ga_h(\bfx^*)} v_h^* \cdot \nabla_{\Ga_h(\bfx^*)} \psi_h \\
    = \int_{\Ga_h(\bfx^*)} \!\!\!\!  g \,\nu_{\Ga_h(\bfx^*)} \cdot \psi_h
    +\int_{\Ga_h(\bfx^*)} \!\!\!\!  d_h \cdot \psi_h.
$$
With $d_h(\cdot,t) = \sum_{j=1}^N d_j(t) \phi_j[\bfx^*(t)]$ and the corresponding nodal vector
$\bfd_{\bfv}(t)=\bigl( d_j(t)\bigr) \in \R^{3N}$ we then have $(I_3\otimes\bfM(\bfx^*(t)))\bfd_{\bfv}(t)$ as the defect on inserting $\bfx^*$ and $\bfv^*$
in the first equation of
\eqref{eq: DAE form - uncoupled}.  With $\bfM^{[3]}(\xs)= I_3\otimes\bfM(\xs)$, we thus have \ecl
\begin{equation}
\label{eq: DAE form - uncoupled - solution}
    \begin{aligned}
        \bfK(\xs) \vs =&\ \bfg(\xs) + \bfM^{[3]}(\xs)\dv, \\
       \dotxs =&\ \bcl \vs \ecl.
    \end{aligned}
\end{equation}
We denote the errors in the surface nodes and in the velocity by $\ex=\bfx-\xs$ and $\ev=\bfv-\vs$, respectively. We rewrite the velocity law in \eqref{eq: DAE form - uncoupled} as
\begin{equation*}
    \bfK(\xs) \bfv = -\big(\bfK(\bfx)-\bfK(\xs)\big) \vs -\big(\bfK(\bfx)-\bfK(\xs)\big)\ev + \bfg(\bfx) .
\end{equation*}
Then, by subtracting \eqref{eq: DAE form - uncoupled - solution} from the above version of 
\eqref{eq: DAE form - uncoupled}, we obtain the following error equations for the uncoupled problem:
\begin{equation}
\label{eq: error equations - uncoupled}
    \begin{aligned}
        \bfK(\xs) \ev =&\ -\big(\bfK(\bfx)-\bfK(\xs)\big)\vs -\big(\bfK(\bfx)-\bfK(\xs)\big) \ev \\
        &\ + \big(\bfg(\bfx) - \bfg(\xs)\big) - \bfM^{[3]}(\xs)\dv , \\
        \dotex =&\ \bcl \ev \ecl .
    \end{aligned}
\end{equation}
When no confusion can arise, we write in the following $\bfM(\xs)$ for $\bfM^{[3]}(\xs)$ and
$\| \cdot \|_{H^1(\Gamma)}$ for $\| \cdot \|_{H^1(\Gamma)^3}$, etc.

\subsection{Norms}
\label{subsection: norms}
We recall that $\bfK(\bfx^*)=I_3\otimes (\bfM(\bfx^*)+\alpha\bfA(\bfx^*))$  and, for $\bfw\in\R^{3N}$ and the corresponding finite element function $w_h= \sum_{j=1}^N w_j \phi_j[\bfx^*] \in S_h(\bfx^*)^3$, we consider the norm
\begin{align*}
\| \bfw \|_{\bfK(\bfx^*)}^2 := & \ \bfw^T \bfK(\bfx^*) \bfw
\\ = & \ \|w_h \|_{L^2(\Gamma_h(\bfx^*))}^2 + \alpha
\| \nabla_{\Gamma_h(\bfx^*)} w_h \|_{L^2(\Gamma_h(\bfx^*))}^2 \sim \| w_h \|_{H^1(\Gamma_h(\bfx^*))}^2.
\end{align*}
For convenience, we will take $\alpha=1$ in the remainder of this section, so that the last norm equivalence becomes an equality.
For the defect $d_h\in S_h(\bfx^*)^3$  we use the dual norm
(cf.~\cite[Proof of Theorem 5.1]{LubichMansourVenkataraman_bdsurf})
\begin{equation}
\label{eq: star and H^-1 norm equality}
    \begin{aligned}
        & \|d_h\|_{H_h\inv(\Gamma_h(\bfx^*))} :=
        \sup_{0\neq \psi_h \in S_h(\bfx^*)^3} \frac{\int_{\Gamma_h(\bfx^*)} d_h \cdot  \psi_h}{\|\psi_h\|_{H^1(\Gamma_h(\bfx^*))^3}}
        \\
        &=  \sup_{0\neq \bfz \in \R^{3N}} \frac{\dv^T \bfM(\xs) \bfz}{(\bfz^T \bfK(\xs) \bfz)^\half}
        = \sup_{0\neq \bfw \in \R^{3N}} \frac{\dv^T \bfM(\xs)\bfK(\xs)^{-\half} \bfw}{(\bfw^T \bfw)^\half}
        \\
        &=\ \| \bfK(\xs)^{-\half} \bfM(\xs)\dv\|_2 =  (\dv^T \bfM(\xs)\bfK(\xs)\inv \bfM(\xs)\dv )^\half.
    \end{aligned}
\end{equation}
We denote
$$
 \|\dv\|_{\star,\xs}^2 := \dv^T \bfM(\xs)\bfK(\xs)\inv \bfM(\xs)\dv ,
$$
so that
\begin{equation*}
  \|\dv\|_{\star,\xs} = \|d_h\|_{H_h\inv(\Gamma_h(\bfx^*))}.
\end{equation*}

\subsection{Stability estimate}
The following stability result holds for the errors $\ev$ and $\ex$, under an assumption of small defects.
It will be shown in Section~\ref{section: defect bounds} that this assumption is satisfied if the exact solution is sufficiently smooth. 
\begin{proposition}
\label{proposition: stability - uncoupled problem}
 \bcl   Suppose that the defect is bounded as follows, with $\kappa>1$:
    \begin{equation*}
       \|\dv(t)\|_{\star,\xs(t)} \leq c h^\kappa ,
        \qquad t\in[0,T] .
    \end{equation*}
    Then there exists  $h_0>0$ such that the following error bounds hold for $h\leq h_0$ and $0\le t \le T$:
    \begin{align}
    \label{ex}
       & \|\ex(t)\|_{\bfK(\xs(t))}^2 \leq C \int_0^t  \|\dv(s)\|_{\star,\xs}^2  \d s,
    \\
    \label{ev}
      &  \|\ev(t)\|_{\bfK(\xs(t))}^2 \le C \|\dv(t)\|_{\star,\xs}^2
        + C \int_0^t  \|\dv(s)\|_{\star,\xs}^2  \d s .
    \end{align}
%
    The constant $C$ is independent of $t$ and $h$, but depends on the final time $T$ and on the  regularization parameter $\alpha$.
\end{proposition}

We note that the error functions $e_v(\cdot,t),e_x(\cdot,t)\in S_h(\xs(t))^3$ with nodal vectors $\ev(t)$ and $\ex(t)$, respectively,
are then bounded by
$$
    \begin{aligned}
        \|e_v(\cdot,t)\|_{H^1(\Gamma_h(\bfx^*(t)))} \leq  Ch^\kappa \andquad
        \|e_x(\cdot,t)\|_{H^1(\Gamma_h(\bfx^*(t)))} \leq  Ch^\kappa ,
    \end{aligned}
    \quad t\in[0,T] .
$$

\begin{proof} \bcl The proof uses energy estimates for the error equations  \eqref{eq: error equations - uncoupled} in the matrix-vector formulation, and it
relies on the results of Section \ref{section: preliminaries}. \ecl In the course of this proof $c$ and $C$ will be generic constants that take on different values on different occurrences.

In view of condition \eqref{e-inf-bound} for $\bfy=\xs(t)$, we will need to control the $W^{1,\infty}$ norm of the position error $e_x(\cdot,t)$. Let $0<t^*\le T$ be the maximal time such that
\begin{equation}
\label{eq: assumed bounds}
        \|\nabla_{\Ga_h(\xs(t))}  e_x(\cdot,t)\|_{L^\infty(\Ga_h(\xs(t)))} \leq  h^{(\kappa-1)/2}
        \\
     \qquad \textrm{ for } \quad t\in[0,t^*].
\end{equation}
At $t=t^*$  either this inequality  becomes an equality, or else we have $t^*=T$.

We will first  prove the stated error bounds for $0\leq t \leq t^*$.
Then the proof will be finished by showing that in fact $t^*$ coincides with $T$.

\smallskip
By testing the first equation in \eqref{eq: error equations - uncoupled} with $\ev$, and dropping the omnipresent argument $t \in [0,t^*]$, we obtain:
\begin{align*}
    \normK{\ev}^2 = \ev^T \bfK(\xs) \ev =&\ -\ev^T\big(\bfK(\bfx)-\bfK(\xs)\big)\vs\\
    &\ -\ev^T \big(\bfK(\bfx)-\bfK(\xs)\big) \ev \\
    &\ + \ev^T \big(\bfg(\bfx) - \bfg(\xs)\big) - \ev^T \bfM(\xs) \dv .
\end{align*}
We separately estimate the four terms on the right-hand side in an appropriate way, with Lemmas~\ref{lemma: matrix differences} -- \ref{lemma: normal vector perturbation} as our main tools.

(i) We denote, for $0\le  \theta\le 1$, by $e_v^\theta$ and $v_h^{*,\theta}$ the finite element functions in $S_h(\Gamma_h^\theta)^3$ for $\Gamma_h^\theta=\Gamma_h(\xs+\theta\ex)$ with nodal vectors $\ev$ and $\bfv^*$, respectively. Lemma~\ref{lemma: matrix differences} then gives us
\begin{align*}
    & \ev^T\big(\bfK(\bfx)-\bfK(\xs)\big)\vs \\
    &= \int_0^1 \int_{\Gamma_h^\theta} e_v^\theta \cdot \bigl( \nabla_{\Gamma_h^\theta}\cdot e_x^\theta \bigr) v_h^{*,\theta}\, \d\theta
    +
    \alpha \int_0^1 \int_{\Gamma_h^\theta} \nabla_{\Gamma_h^\theta} e_v^\theta \cdot \bigl( D_{\Gamma_h^\theta} e_x^\theta \bigr) \nabla_{\Gamma_h^\theta}v_h^{*,\theta}\, \d\theta.
\end{align*}
Using the Cauchy--Schwarz inequality, we estimate the integral with the product of the $L^2-L^2-L^\infty$ norms of the three factors. We thus have
\begin{align*}
    \ev^T\bigl(\bfK(\bfx)-\bfK(\xs)\bigr)\vs
    &\le \int_0^1 \| e_v^\theta \|_{L^2(\Gamma_h^\theta)} \,
    \| \nabla_{\Gamma_h^\theta}\cdot e_x^\theta \|_{L^2(\Gamma_h^\theta)} \,
    \| v_h^{*,\theta} \|_{L^{\infty}(\Gamma_h^\theta)}\, \d\theta
    \\ & \quad + \alpha
    \int_0^1 \| \nabla_{\Gamma_h^\theta} e_v^\theta \|_{L^2(\Gamma_h^\theta) }\, \| D_{\Gamma_h^\theta} e_x^\theta \|_{L^2(\Gamma_h^\theta)} \,
    \| \nabla_{\Gamma_h^\theta}v_h^{*,\theta} \|_{L^{\infty}(\Gamma_h^\theta)}\, \d\theta
    \\
    & \le c \int_0^1 \| e_v^\theta \|_{H^1(\Gamma_h^\theta)} \, \| e_x^\theta \|_{H^1(\Gamma_h^\theta)} \,
    \| v_h^{*,\theta} \|_{W^{1,\infty}(\Gamma_h^\theta)}\, \d\theta.
\end{align*}
By \eqref{eq: assumed bounds} and Lemma~\ref{lemma:theta-independence}, this is bounded by
\begin{align*}
    \ev^T\big(\bfK(\bfx)-\bfK(\xs)\big)\vs \le
    &\ c \| e_v \|_{H^1(\Gamma_h(\xs))} \, \| e_x\|_{H^1(\Gamma_h(\xs))} \,
    \| v_h^* \|_{W^{1,\infty}(\Gamma_h(\xs))} ,
\end{align*}
where the last factor is bounded independently of $h$.
By the Young inequality, we thus obtain
\begin{align*}
    \ev^T \big(\bfK(\bfx)-\bfK(\xs)\big)\vs
    \leq & \ \tfrac{1}{6}\| e_v \|_{H^1(\Gamma_h(\xs))}^2 +
    C \| e_x \|_{H^1(\Gamma_h(\xs))}^2
    \\
    = &\  \tfrac{1}{6} \normK{\ev}^2 + C \normK{\ex}^2 .
\end{align*}

(ii) Similarly, estimating the three factors in the integrals by $L^2-L^\infty-L^2$, we obtain
\begin{align*}
    \ev^T \big(\bfK(\bfx)-\bfK(\xs)\big)\ev
    \leq &\ c \|e_\bfv\|_{L^2(\Ga_h(\xs))}^2 \|\nbgh \cdot e_\bfx\|_{L^\infty(\Ga_h(\xs))} \\
    &\ + c \alpha \|\nbgh e_\bfv\|_{L^2(\Ga_h(\xs))}^2 \|D_{\Ga_h} e_\bfx\|_{L^\infty(\Ga_h(\xs))} \\
    \leq &\  c h^{(\kappa-1)/2} \normK{\ev}^2,
\end{align*}
where in the last inequality we used the  bound \eqref{eq: assumed bounds}.

(iii) In the following bound we  use Lemma~\ref{lemma: mat normal vector}.
Again with the finite element function $e_v^\theta= \sum_{j=1}^N (\ev)_j \phi_j[\xs+\theta\ex]$ on the surface $\Gamma_h^\theta=\Gamma_h(\xs+\theta\ex)$, for~$0\le\theta\le 1$, we write
\begin{align*}
    &\ \ev^T \big(\bfg(\bfx) - \bfg(\xs)\big)  = \int_{\Ga_h^1} g \nu_{\Ga_h^1} \cdot e_{v}^1 - \int_{\Ga_h^0} g \nu_{\Ga_h^0} \cdot e_{v}^0
 = \int_0^1 \frac\d{\d\theta}  \int_{\Ga_h^\theta} g \nu_{\Ga_h^\theta} \cdot e_{v}^\theta \d\theta.
\end{align*}
Using the Leibniz formula, this becomes
$$
 \ev^T \big(\bfg(\bfx) - \bfg(\xs)\big) =
 \int_0^1  \int_{\Ga_h^\theta} \Bigl(\mat_\theta\bigl( g \nu_{\Ga_h^\theta}  \cdot e_{v}^\theta \bigr) +
 (g \nu_{\Ga_h^\theta}  \cdot e_{v}^\theta ) (\nabla_{\Ga_h^\theta}\cdot e_x^\theta) \Bigr)\, \d\theta.
 $$
Here we have, noting that $\mat_\theta e_{v}^\theta=0$,
$$
    \mat_\theta\bigl( g \nu_{\Ga_h^\theta}  \cdot e_{v}^\theta \bigr) =
    g' e_x^\theta \,\nu_{\Ga_h^\theta}  \cdot e_{v}^\theta +
    g \,\mat_\theta \nu_{\Ga_h^\theta}  \cdot e_{v}^\theta.
$$
With Lemmas~\ref{lemma:theta-independence} and \ref{lemma: mat normal vector} we therefore obtain via the Cauchy-Schwarz inequality
\begin{align*}
    \int_{\Ga_h^\theta} \mat_\theta\bigl( g \nu_{\Ga_h^\theta}  \cdot e_{v}^\theta \bigr)  \leq &\ 
    \bcl c_2^2 \,\ecl \| g' \|_{L^\infty} \, \| e_x \|_{L^2(\Ga_h(\xs))} \|  e_v \|_{L^2(\Ga_h(\xs))} \\
    &\ + \bcl c_2^2 \,\ecl \| g \|_{L^\infty} \, \|  \nabla_{\Ga_h(\xs)} e_x \|_{L^2(\Ga_h(\xs))} \, \|  e_v \|_{L^2(\Ga_h(\xs))} ,
\end{align*}
and again with Lemma~\ref{lemma:theta-independence},
$$
    \int_{\Ga_h^\theta} (g \nu_{\Ga_h^\theta}  \cdot e_{v}^\theta ) (\nabla_{\Ga_h^\theta}\cdot e_x^\theta) \le \bcl c_2^2 \,\ecl
    \| g \|_{L^\infty} \, \|  e_v \|_{L^2(\Ga_h(\xs))}  \, \|  \nabla_{\Ga_h(\xs)} \cdot e_x \|_{L^2(\Ga_h(\xs))}.
$$
In total, we obtain a bound of the same type as for the terms in (i) and (ii):
\begin{align*}
\ev^T \big(\bfg(\bfx) - \bfg(\xs)\big) &\le c  \|   e_x \|_{H^1(\Ga_h(\xs))} \, \|  e_v \|_{L^2(\Ga_h(\xs))}
\\
&= c \normK{\ex} \, \normM{\ev} \le  \tfrac{1}{6} \normK{\ev}^2 + C \normK{\ex}^2 .
\end{align*}

The combination of the estimates of the three terms (i)--(iii) with absorptions (for sufficiently small $h\leq h_0$), and a simple dual norm estimate, based on \eqref{eq: star and H^-1 norm equality},  for the defect term, yield the bound
\begin{equation}
\label{eq: velocity stability bound}
    \begin{aligned}
        \normK{\ev}^2
        \leq &\ c \normK{\ex}^2 + c \|\dv\|_{\star,\xs}^2.  \\
    \end{aligned}
\end{equation}
Using this estimate, together with taking the $\normK{\cdot}$ norm of both sides of the second equation in \eqref{eq: error equations - uncoupled}, we obtain
\begin{equation}
\label{eq: pre-final energy bound}
        \bcl \normK{\dotex}^2 = \ \normK{\ev}^2  
        \leq \ c \normK{\ex}^2 + c \|\dv\|_{\star,\xs}^2  .
\end{equation}

In order to apply Gronwall's inequality, we connect $\displaystyle\diff \!\normK{\ex}^2$ and $\normK{\dotex}^2$ as follows:
\begin{align*}
    \half \diff \normK{\ex}^2 =&\ \ex^T \bfK(\xs) \dotex + \half \ex^T \Big(\diff \bfK(\xs)\Big) \ex \\
    \leq &\ \normK{\dotex}^2 + c \normK{\ex}^2,
\end{align*}
where we use the Cauchy-Schwarz inequality and Lemma~\ref{lemma: matrix derivatives} in the estimate.
Inserting \eqref{eq: pre-final energy bound}, we obtain
\begin{equation*}
    \half \diff \normK{\ex}^2 \leq \bcl c \normK{\ex}^2 + c \|\dv\|_{\star,\xs}^2  . \ecl
\end{equation*}
A Gronwall inequality then yields \eqref{ex},
using $e_j(0) = x_j(0) - x_j^0 = 0$ for $j=1,\dotsc,N$.
Inserting this estimate in \eqref{eq: velocity stability bound},
we can bound $\ev(t)$  for $0\le t \le t^*$ by \eqref{ev}.

\smallskip
Now it only remains to show that $t^*=T$ for $h$ sufficiently small. For $0\leq t \leq t^*$ we use an inverse inequality and \eqref{ex} to bound the left-hand side in \eqref{eq: assumed bounds}:
\begin{align*}
    \|\nabla_{\Ga_h(\xs(t))}  e_x(\cdot,t)\|_{L^\infty(\Ga_h(\xs(t)))}    \leq &\ c h\inv \|\nabla_{\Ga_h(\xs(t))}  e_x(\cdot,t)\|_{L^2(\Ga_h(\xs(t)))} \\
    \leq &\ c h\inv \|\ex(t)\|_{\bfK(\xs(t))}
    \leq  c C h^{\kappa-1} \leq \tfrac12 h^{(\kappa-1)/2}
\end{align*}
for sufficiently small $h$.  Hence, we can extend the bound \eqref{eq: assumed bounds} beyond $t^*$, which contradicts the maximality of $t^*$ unless we have already $t^*=T$.
\QED \end{proof}

\section{Stability of coupling surface PDEs to surface motion}
\label{section: coupled problem}

Now we turn to the stability bounds of the original problem \eqref{uh-vh-equation}--\eqref{xh}, or in DAE form \eqref{eq: DAE form - coupled}, which is the formulation we will actually use for the stability analysis.

\subsection{Error equations}

Similarly as before, in order to \bcl derive \ecl stability estimates we consider the DAE system when we insert the nodal values  $\us\bcl(t)\ecl\in\R^{N}$  of the exact solution $u\bcl(\cdot,t)\ecl$, the nodal values $\xs\bcl(t)\ecl\in\R^{3N}$ of the exact positions $X\bcl(\cdot,t)\ecl$, and the  nodal values $\vs\bcl(t)\ecl\in\R^{3N}$ of the exact velocity  $v\bcl(\cdot,t)\ecl$. Inserting them into \eqref{eq: DAE form - coupled} yields defects \bcl $\du\bcl(t)\in \R^N\ecl$ and $\dv\bcl(t)\in \R^{3N}\ecl$: omitting the argument $t$ in the notation, we have\ecl
\begin{equation}
\label{eq: DAE form - coupled - solution}
    \begin{aligned}
        \diff \Big(\bfM(\xs)\us\Big) + \bfA(\xs)\us =&\ \bff(\xs,\us) + \bfM(\xs)\du, \\
        \bfK(\xs)\vs =&\ \bfg(\xs,\us) + \bfM^{[3]}(\xs)\dv, \\
        \dotxs =&\ \bcl \vs , \ecl
    \end{aligned}
\end{equation}
\bcl
where again $\bfM^{[3]}(\xs) = I_3 \otimes \bfM(\xs)$.
As no confusion can arise, we write again $\bfM(\xs)$ for $\bfM^{[3]}(\xs)$. \ecl

We denote the PDE error by $\eu=\bfu-\us$, and \bcl as in the previous section, \ecl $\ev=\bfv-\vs$ and $\ex=\bfx-\xs$ denote the velocity error and surface error, respectively. Subtracting \eqref{eq: DAE form - coupled - solution} from \eqref{eq: DAE form - coupled}, we obtain the following error equation:
\begin{equation}
\label{eq: error equations - coupled}
    \begin{aligned}
        \diff \Big(\bfM(\xs)\eu\Big) + \bfA(\xs)\eu =&\
        - \diff \Big(\big( \bfM(\bfx)-\bfM(\xs) \big) \us \Big)\\
        &\ - \diff \Big( \big( \bfM(\bfx)-\bfM(\xs) \big) \eu\Big) \\
        &\ - \big( \bfA(\bfx)-\bfA(\xs) \big) \us \\
        &\ - \big( \bfA(\bfx)-\bfA(\xs) \big) \eu \\
        &\ + \big(\bff(\bfx,\bfu) - \bff(\xs,\us)\big) - \bfM(\xs)\du , \\
        \bfK(\xs) \ev =&\ -\big(\bfK(\bfx)-\bfK(\xs)\big)\vs - \big(\bfK(\bfx)-\bfK(\xs)\big) \ev \\
        &\ + \big(\bfg(\bfx,\bfu) - \bfg(\xs,\us)\big) - \bfM(\xs)\dv, \\
        \dotex =&\ \bcl \ev \ecl .
    \end{aligned}
\end{equation}

\subsection{Stability estimate}
\label{subsection: stability estimate - PDE}
We now formulate the stability result for the errors $\eu$, $\ev$ and $\ex$ of the surface motion coupled to the surface PDE. \bcl Here, we use the norms \eqref{M-L2}-\eqref{A-H1} and those of Section~\ref{subsection: norms}. \ecl
\bcl
\begin{proposition}
\label{proposition: stability - coupled problem}
    Assume that the following bounds hold for the defects, for some $\kappa>1$: 
    \begin{equation*}
            \|\du(t)\|_{\star,\xs(t)} \leq  c h^\kappa , \quad\ 
            \|\dv(t)\|_{\star,\xs(t)} \leq  c h^\kappa , 
        \qquad \hbox{for }\  t\in[0,T] .
    \end{equation*}
    Then there exists  $h_0>0$ such that the following stability estimate holds for all $h\leq h_0$ and $0\le t \le T$:
    \begin{align}
\nonumber
        \normM{\eu(t)}^2 + &\ \int_0^t \normA{\eu(s)}^2 \d s + \normK{\ex(t)}^2 +\int_0^t \normK{\ev(s)}^2 \d s
        \\  
        \label{eq: stability estimate}
        \leq &\ C \int_0^t \Bigl(\|\du(s)\|_{\star,\xs}^2 + \|\dv(s)\|_{\star,\xs}^2  \Bigr)\d s .
  \end{align}
%
%
%
    The constant $C$ is independent of $t$ and $h$, but depends on the final time $T$ and on the regularization parameter $\alpha$.
\end{proposition}

We note that the error functions $e_u(\cdot,t)\in S_h(\xs(t))$ and $e_v(\cdot,t),e_x(\cdot,t)\in S_h(\xs(t))^3$ with nodal vectors $\eu(t)$ and $\ev(t),\ex(t)$, respectively,
are then bounded by
\begin{align} 
      \nonumber
      &  \|e_u(\cdot,t)\|_{L^2(\Gamma_h(\bfx^*(t)))} + \biggl( \int_0^t  \|e_u(\cdot,t)\|_{H^1(\Gamma_h(\bfx^*(t)))}^2 \, \d s \biggr)^{1/2} \leq  Ch^\kappa, \\
      \label{err-cont-norms}
      &  \biggl( \int_0^t \|e_v(\cdot,t) \|_{H^1(\Gamma_h(\bfx^*(t)))^3}^2  \d s \biggr)^{1/2}\leq  Ch^\kappa , \\
      \nonumber
     &   \|e_x(\cdot,t)\|_{H^1(\Gamma_h(\bfx^*(t)))^3} \leq  Ch^\kappa , 
     \qquad\qquad  \qquad\qquad\qquad   t\in[0,T] . 
\end{align}
\ecl

\begin{proof}
The proof is an extension of the proof of Proposition~\ref{proposition: stability - uncoupled problem}, again based on the matrix-vector formulation and  the auxiliary results of Section \ref{section: preliminaries}.
We handle the surface PDE and the surface equations separately: we first estimate the errors \bcl of \ecl the PDE, while those for the surface equation are based on Section~\ref{section: uncoupled problem}. Finally we will combine the results to obtain the stability estimates for the coupled problem.
\bcl
In the course of this proof $c$ and $C$ will be generic constants that take on different values on different occurrences.
\ecl

Let $0<t^*\leq T$ be the maximal time such that the following inequalities hold:
\begin{equation}
\label{eq: assumed bounds - 2}
    \begin{aligned}
        \|\nabla_{\Ga_h(\xs(t))}\bcl e_x(\cdot,t)\ecl \|_{L^\infty(\Ga_h(\xs(t)))} \leq &\ h^{(\kappa-1)/2} , \\
        \|\bcl e_u(\cdot,t) \ecl\|_{L^\infty(\Ga_h(\xs(t)))} \leq &\ 1,
    \end{aligned} \qquad \textrm{ for } \quad t\in[0,t^*].
\end{equation}
Note that $t^*>0$ since initially both \bcl $e_x(\cdot,0)=0$ and $e_u(\cdot,0)=0$. \ecl

We first prove the  stated error bounds for $0\leq t \leq t^*$.
At the end, the proof will be finished by showing that in fact $t^*$ coincides with $T$.

\smallskip
\newcommand{\test}{\eu}
Testing the first two equations of \eqref{eq: error equations - coupled} with $\test$ and $\ev$, and dropping the omnipresent argument $t \in [0,t^*]$, we obtain:
\begin{equation*}
    \begin{aligned}
        \! \! \test^T \! \diff \! \Big(\! \bfM(\xs)\eu\! \Big) + \test^T \bfA(\xs)\eu =&
        - \test^T \diff \Big( \big( \bfM(\bfx)-\bfM(\xs) \big) \us \Big) \\
        & - \test^T \diff \Big( \big( \bfM(\bfx)-\bfM(\xs) \big) \eu\Big) \\
        & - \test^T \big( \bfA(\bfx)-\bfA(\xs) \big) \us \\
        & - \test^T \big( \bfA(\bfx)-\bfA(\xs) \big) \eu \\
        & + \test^T \big(\bff(\bfx,\bfu) - \bff(\xs,\us)\big) - \test^T \bfM(\xs)\du , \\
        \normK{\ev}^2  =& - \ev^T \big(\bfK(\bfx)-\bfK(\xs)\big)\vs \! - \ev^T \big(\bfK(\bfx)-\bfK(\xs)\big) \ev \\
        & + \ev^T \big(\bfg(\bfx,\bfu) - \bfg(\xs,\us)\big) - \ev^T \bfM(\xs)\dv , \\
        \dotex =& \ev  .
    \end{aligned}
\end{equation*}
(A) {\it Estimates for the surface PDE:}
We estimate the terms  separately, \bcl with Lemmas~\ref{lemma: matrix differences} -- \ref{lemma:theta-independence} \ecl as our main tools.

(i) The symmetry of $\bfM(\xs)$ and a simple calculation yield
\begin{align*}
    \test^T \diff \Big(\bfM(\xs)\eu\Big) = &\ \half \diff \Big(\test^T \bfM(\xs)\eu\Big) + \half \test^T \Big(\diff \bfM(\xs)\Big)\eu \\
    =&\ \half \diff \normM{\test}^2 + \half \test^T \Big(\diff \bfM(\xs)\Big)\eu ,
\end{align*}
where the last term is bounded by Lemma~\ref{lemma: matrix derivatives} as
$$
\Bigl|\eu^T \diff \bfM(\xs)\eu\Bigr| \le c \,\normM{\eu}^2.
$$

(ii) By the definition of the $\bfA$-norm we have
\begin{align*}
    \test^T \bfA(\xs)\eu = \normA{\eu}^2 .
\end{align*}

\bcl
(iii) With the product rule we write
\begin{align}\nonumber
&\test^T \diff \Big( \big( \bfM(\bfx)-\bfM(\xs) \big) \us \Big)
\\
\label{term-iii}
&\qquad = \test^T  \big( \bfM(\bfx)-\bfM(\xs) \big) \dot\bfu^*  +
\test^T \Big( \diff \big( \bfM(\bfx)-\bfM(\xs) \big) \Big) \bfu^*.
\end{align}
With  $\Gamma_h^\theta(t)=\Gamma_h[\xs(t)+\theta \bfe_\bfx(t)]$ and with the finite element functions
$$
\hbox{
$e_u^\theta(\cdot,t)$, $u_h^{*,\theta}(\cdot,t)\in S_h(\xs(t)+\theta \ex(t))$ with nodal vectors $\eu(t)$, $\bfu^*(t)$, resp.,  
}
$$
Lemma~\ref{lemma: matrix differences} (with $\xs(t)$ in the role of $\bfy$) yields for the first term, omitting again the argument $t$,
$$
\test^T  \big( \bfM(\bfx)-\bfM(\xs) \big) \dot\bfu^* = \int_0^1 \int_{\Gamma_h^\theta}
e_u^\theta\, (\nabla_{\Gamma_h^\theta} \cdot e_x^\theta) \, \partial_h^\bullet u_h^{*,\theta}\, \d\theta.
$$
Using the Cauchy-Schwarz inequality we obtain 
$$
|\test^T  \big( \bfM(\bfx)-\bfM(\xs) \big) \dot\bfu^*|
\le \int_0^1 \| e_u^\theta \|_{L^2(\Gamma_h^\theta)}\, \| \nabla_{\Gamma_h^\theta} \cdot e_x^\theta \|_{L^2(\Gamma_h^\theta)}\, \| \partial_h^\bullet u_h^{*,\theta} \|_{L^\infty(\Gamma_h^\theta)}\, \d\theta.
$$
Under condition \eqref{eq: assumed bounds - 2} we obtain from
Lemmas~\ref{lemma:cond-equiv} and~\ref{lemma:theta-independence}
that for $0\le t \le t^*$,
$$
|\test^T  \big( \bfM(\bfx)-\bfM(\xs) \big) \dot\bfu^* |
\le  c \, \| e_u^0 \|_{L^2(\Gamma_h^0)} \,
\| e_x^0 \|_{H^1(\Gamma_h^0)} \, \| \partial_h^\bullet u_h^{*,0} \|_{L^\infty(\Gamma_h^0)}.
$$
Now, the last factor is bounded by
$$
\| \partial_h^\bullet u_h^{*,0} \|_{L^\infty(\Gamma_h^0)} \le c \| \dot\bfu^* \|_\infty \le C
$$
because of the assumed smoothness of the exact solution $u$ and hence of its material derivative $\partial^\bullet u(\cdot,t)$, whose values at the nodes are the entries of the vector $\dot\bfu^*(t)$.
Hence we obtain, on recalling the definitions of the discrete norms,
$$
-\test^T  \big( \bfM(\bfx)-\bfM(\xs) \big) \dot\bfu^* \le C  \normM{\eu}  \normK{\ex}.
$$
Using Lemma~\ref{lemma: matrix differences} together with the Leibniz formula, the last term in \eqref{term-iii} becomes
\begin{align*}
\test^T \Big( \diff \big( \bfM(\bfx)-\bfM(\xs) \big) \Big) \us &= \int_0^1 \int_{\Gamma_h^\theta}
e_u^\theta\, \partial_h^\bullet (\nabla_{\Gamma_h^\theta} \cdot e_x^\theta) \,  u_h^{*,\theta}\, \d\theta
\\
&+
 \int_0^1 \int_{\Gamma_h^\theta}
e_u^\theta\, (\nabla_{\Gamma_h^\theta} \cdot e_x^\theta) \,  u_h^{*,\theta}\, (\nabla_{\Gamma_h^\theta} \cdot v_h^\theta)\, \d\theta,
\end{align*}
where $v_h^\theta$ is the velocity of $\Gamma_h^\theta$ (as a function of $t$), which is the finite element function in
$S_h(\xs+\theta \ex)$ with nodal vector $\dot \bfx^* + \theta \dot\bfe_\bfx = \bfv^* + \theta \ev$, so that
\begin{equation}\label{vhtheta}
v_h^\theta = v_h^{*,\theta} + \theta e_v^\theta,
\end{equation}
where $v_h^{*,\theta}$ and $e_v^\theta$ are the finite element functions on $\Gamma_h^\theta$ with nodal vectors $\bfv^*$ and $\ev$, respectively.
In the first integral we further use, cf.\ \cite[Lemma~2.6]{DziukKronerMuller}, 
$$
\partial_h^\bullet (\nabla_{\Gamma_h^\theta} \cdot e_x^\theta) = \nabla_{\Gamma_h^\theta} \cdot \partial_h^\bullet e_x^\theta - \bigl((I_3-\nu_h^\theta(\nu_h^\theta)^T) \nabla_{\Gamma_h^\theta} v_h^\theta\bigr) : \nabla_{\Gamma_h^\theta} e_x^\theta ,
$$
where $:$ symbolizes the Euclidean inner product of the vectorization of two matrices. Here we note that $\partial_h^\bullet e_x^\theta$ is the finite element function on $\Gamma_h^\theta$ with nodal vector $\dot\bfe_\bfx=\ev$, so that $\partial_h^\bullet e_x^\theta=e_v^\theta$. 

We then estimate, using the Cauchy-Schwarz inequality in the first step, Lemmas~\ref{lemma:cond-equiv} and~\ref{lemma:theta-independence} in the second step (using \eqref{eq: assumed bounds - 2} to ensure the smallness condition in these lemmas), the definition of the discrete norms in the third step, and using the first bound of \eqref{eq: assumed bounds - 2}  and the boundedness of the discrete gradient of the interpolated exact velocity $\nabla_{\Ga_h(\xs)} v_h^*$ and of the interpolated exact solution $u_h^*$ in the fourth step,
\begin{align*}
&\biggl|\int_0^1 \int_{\Gamma_h^\theta}
e_u^\theta\, \partial_h^\bullet (\nabla_{\Gamma_h^\theta} \cdot e_x^\theta) \,  u_h^{*,\theta}\, \d\theta \biggr|\\
&\le \int_0^1 \int_{\Gamma_h^\theta} \|e_u^\theta\|_{L^2(\Gamma_h^\theta)}\,
\Bigl(\| \nabla_{\Gamma_h^\theta} \cdot e_v^\theta \|_{L^2(\Gamma_h^\theta)} 
\\
&\hspace{3.5cm}
+
\|  \nabla_{\Gamma_h^\theta} v_h^{*,\theta} \|_{L^\infty(\Gamma_h^\theta)}\cdot
\| \nabla_{\Gamma_h^\theta} e_x^\theta \|_{L^2(\Gamma_h^\theta)}
\\
&\hspace{3.5cm}
+\| \nabla_{\Gamma_h^\theta} e_v^\theta \|_{L^2(\Gamma_h^\theta)} \cdot
\| \nabla_{\Gamma_h^\theta} e_x^\theta \|_{L^\infty(\Gamma_h^\theta)} \Bigr)
\| u_h^{*,\theta} \|_{L^\infty(\Gamma_h^\theta)} \, \d\theta
\\
& \le c\, \|  e_u \|_{L^2(\Gamma_h(\xs))} \Bigl(
 \| \nabla_{\Gamma_h(\xs)} e_v \|_{L^2(\Gamma_h(\xs))} 
 \\
 &\hspace{3cm}
 +
 \|  \nabla_{\Gamma_h(\xs)} v_h^* \|_{L^\infty(\Gamma_h(\xs))} \cdot
 \| \nabla_{\Gamma_h(\xs)} e_x \|_{L^2(\Gamma_h(\xs))} 
 \\
 &\hspace{2cm}
 + \| \nabla_{\Gamma_h(\xs)} e_v \|_{L^2(\Gamma_h(\xs))} \cdot
  \| \nabla_{\Gamma_h(\xs)} e_x \|_{L^\infty(\Gamma_h(\xs))} 
 \Bigr)
  \|  u_h^* \|_{L^\infty(\Gamma_h(\xs))} 
\\
&\le c \,\normM{\eu}  \Bigl( \normA{\ev} + \|\nabla_{\Ga_h(\xs)} v_h^* \|_{L^\infty(\Ga_h(\xs))} \, \normA{\ex} 
\\
&\hspace{4.5cm}
+ \normA{\ev}     
\|\nabla_{\Ga_h(\xs)} e_x \|_{L^\infty(\Ga_h(\xs))} \Bigr) \| \bfu^* \|_\infty
\\
&\le c  \normM{\eu} \Bigl( \normK{\ev} + C \normK{\ex} +
 \normK{\ex} h^{(\kappa-1)/2} \Bigr) C
\\
&\le C' \normM{\eu} \Bigl( \normK{\ev} + \normK{\ex} \Bigr).
\end{align*}
With the same arguments we estimate, on inserting \eqref{vhtheta},
\begin{align*}
 &\biggl|\int_0^1 \int_{\Gamma_h^\theta}
e_u^\theta\, (\nabla_{\Gamma_h^\theta} \cdot e_x^\theta) \,  u_h^{*,\theta}\, (\nabla_{\Gamma_h^\theta} \cdot v_h^\theta)\, \d\theta \biggr|
\\
&\le  \int_0^1 \int_{\Gamma_h^\theta} \| e_u^\theta\|_{L^2(\Gamma_h^\theta)}\, \|\nabla_{\Gamma_h^\theta} \cdot e_x^\theta\|_{L^2(\Gamma_h^\theta)} \,  \|u_h^{*,\theta}\|_{L^\infty(\Gamma_h^\theta)}\, 
\|\nabla_{\Gamma_h^\theta} \cdot v_h^{*,\theta}\|_{L^\infty(\Gamma_h^\theta)}\, \d\theta
\\
&\ + \int_0^1 \int_{\Gamma_h^\theta} \| e_u^\theta\|_{L^2(\Gamma_h^\theta)}\, \|\nabla_{\Gamma_h^\theta} \cdot e_x^\theta\|_{L^\infty(\Gamma_h^\theta)} \,  \|u_h^{*,\theta}\|_{L^\infty(\Gamma_h^\theta)}\, 
\|\nabla_{\Gamma_h^\theta} \cdot e_v^{\theta}\|_{L^2(\Gamma_h^\theta)}\, \d\theta
\\
&\le c\, \normM{\eu} \normK{\ex} \| \bfu^* \|_\infty \|\nabla_{\Ga_h(\xs)}\cdot v_h^* \|_{L^\infty(\Ga_h(\xs))}
\\
&\ + c\, \normM{\eu}\,\|\nabla_{\Ga_h(\xs)} e_x \|_{L^\infty(\Ga_h(\xs))} \,
\| \bfu^* \|_\infty \normK{\ev}
\\
&\le C\, \normM{\eu} \Bigl( \normK{\ev} + \normK{\ex} \Bigr).
\end{align*}
Altogether we obtain the bound
$$
-\test^T \diff \Big( \big( \bfM(\bfx)-\bfM(\xs) \big) \us \Big) \le C\, \normM{\eu} \Bigl( \normK{\ev} + \normK{\ex} \Bigr).
$$

(iv) We obtain similarly 
\begin{align*}
    &-\test^T  \diff \Big( \big( \bfM(\bfx)-\bfM(\xs) \big) \eu\Big)
    \\
    &= \ -\half \test^T \Big( \diff \big( \bfM(\bfx)-\bfM(\xs) \big)\Big) \eu  - \half \diff \Big( \test^T \big( \bfM(\bfx)-\bfM(\xs) \big) \eu\Big)\\
    &\leq c\, \normM{\eu} \Bigl( \normK{\ev} + \normK{\ex} \Bigr) \|e_u\|_{L^\infty(\Ga_h(\xs))}
    \\
    &\qquad\qquad
    - \half \diff \Big( \test^T \big( \bfM(\bfx)-\bfM(\xs) \big) \eu\Big) \\
    &\leq \ C \normM{\eu}\bigl(  \normK{\ev} +  \normK{\ex}  \bigr)
     - \half \diff \Big( \test^T \big( \bfM(\bfx)-\bfM(\xs) \big) \eu\Big) ,
\end{align*}
where  we used the second bound of \eqref{eq: assumed bounds - 2} in the last inequality.

(v) Lemma~\ref{lemma: matrix differences}, the Cauchy-Schwarz inequality and Lemma~\ref{lemma:theta-independence}  yield
\begin{align*}
  & - \test^T \big( \bfA(\bfx)-\bfA(\xs) \big) \us 
  \\
  & = 
   -\int_0^1 \int_{\Gamma_h^\theta} \nabla_{\Gamma_h^\theta} e_u^\theta \cdot
   \bigl( D_{\Gamma_h^\theta} e_x^\theta \bigr) \nabla_{\Gamma_h^\theta} u_h^{*,\theta} \, \d\theta
  \\
    &\leq c \normA{\eu} \,\normA{\ex} \,  
    \|  \nabla_{\Gamma_h(\xs)} u_h^* \|_{L^\infty(\Gamma_h(\xs))}
  \\
   &\leq C \normA{\eu} \,\normK{\ex}.
\end{align*}

(vi) Similarly  we estimate
\begin{align*}
   - \test^T \big( \bfA(\bfx)-\bfA(\xs) \big) \eu \leq &\ c \normA{\eu}^2 \|D_{\Ga_h(\xs)} e_x\|_{L^\infty(\Ga_h(\xs))} \\
    \leq &\ C h^{(\kappa-1)/2} \normA{\eu}^2,
\end{align*}
where we used the first bound of \eqref{eq: assumed bounds - 2}.

(vii) The coupling term is estimated similarly to (iii) in the proof of Proposition~\ref{proposition: stability - uncoupled problem}:
\begin{align*}
    \test^T \big(\bff(\bfx,\bfu) - \bff(\xs,\us)\big) = &\ \int_{\Ga_h^1} \!  f(u_h,\nabla_{\Ga_h^1} u_h) e_u^1- \int_{\Ga_h^0} \!  f(u_h^\ast,\nabla_{\Ga_h^0} u_h^\ast)e_u^0 .
\end{align*}
With
\begin{equation}\label{uhtheta}
u_h^\theta = \sum_{j=1}^N (u_j^*+\theta (\eu)_j) \, \phi_j[\xs+\theta \ex] = u_h^{*,\theta}+\theta e_u^\theta
\end{equation}
we therefore have
\begin{align*}
    \test^T \big(\bff(\bfx,\bfu) - \bff(\xs,\us)\big) = &\
    \int_0^1 \frac\d{\d\theta}\int_{\Gamma_h^\theta} f(u_h^\theta, \nabla_{\Gamma_h^\theta}u_h^\theta) \, e_u^\theta\, \d\theta
    \end{align*}
and with the Leibniz formula (noting that $e_x^\theta$ is the velocity of the surface $\Gamma_h^\theta$ considered as a function of $\theta$), we rewrite this as
\begin{align*}
   & \test^T \big(\bff(\bfx,\bfu) - \bff(\xs,\us)\big) \\
    & = \
    \int_0^1\int_{\Gamma_h^\theta} \Bigl( \partial_\theta^\bullet f(u_h^\theta, \nabla_{\Gamma_h^\theta}u_h^\theta) \, e_u^\theta + f(u_h^\theta, \nabla_{\Gamma_h^\theta}u_h^\theta) \, e_u^\theta \,(\nabla_{\Gamma_h^\theta}\cdot e_x^\theta) \Bigr)\d\theta.
    \end{align*}
 Here we use the chain rule
 $$
  \partial_\theta^\bullet f(u_h^\theta, \nabla_{\Gamma_h^\theta}u_h^\theta) =
  \partial_1 f (u_h^\theta, \nabla_{\Gamma_h^\theta}u_h^\theta) \,  \partial_\theta^\bullet u_h^\theta +
  \partial_2 f (u_h^\theta, \nabla_{\Gamma_h^\theta}u_h^\theta) \,  \partial_\theta^\bullet  \nabla_{\Gamma_h^\theta}u_h^\theta
 $$
 \bcl 
 and observe the following: by the assumed smoothness of $f$ and the exact solution $u$, and by the bound~\eqref{eq: assumed bounds - 2} for $e_u$ (and hence for $e_u^\theta$ by 
 Lemmas~\ref{lemma:cond-equiv} and~\ref{lemma:theta-independence}), 
 we have on recalling \eqref{uhtheta}
$$
\|  \partial_i f (u_h^\theta, \nabla_{\Gamma_h^\theta}u_h^\theta) \|_{L^\infty(\Gamma_h^\theta)} \le C, \qquad i=1,2.
$$
We note
 $$
  \partial_\theta^\bullet u_h^\theta = e_u^\theta
 $$
and the relation, see \cite[Lemma~2.6]{DziukKronerMuller},
\bcl
 $$
  \partial_\theta^\bullet  \nabla_{\Gamma_h^\theta}u_h^\theta =
   \nabla_{\Gamma_h^\theta} \partial_\theta^\bullet u_h^\theta  
   -  \nabla_{\Gamma_h^\theta} e_x^\theta\, \nabla_{\Gamma_h^\theta}u_h^\theta
   + \nu_h^\theta (\nu_h^\theta)^T (\nabla_{\Gamma_h^\theta} e_x^\theta)^T \nabla_{\Gamma_h^\theta}u_h^\theta.
 $$
We then have, on inserting \eqref{uhtheta} and using once again
Lemmas~\ref{lemma:cond-equiv} and~\ref{lemma:theta-independence} and the bound~\eqref{eq: assumed bounds - 2},
\begin{align*}
& \test^T \big(\bff(\bfx,\bfu) - \bff(\xs,\us)\big)
\\
&= \int_0^1\int_{\Gamma_h^\theta} e_u^\theta \Bigl( \partial_1 f (u_h^\theta, \nabla_{\Gamma_h^\theta}u_h^\theta) e_u^\theta 
\\
&\quad +  \partial_2 f (u_h^\theta, \nabla_{\Gamma_h^\theta}u_h^\theta)
\bigl( \nabla_{\Gamma_h^\theta} e_u^\theta  
-  \nabla_{\Gamma_h^\theta} e_x^\theta\, \nabla_{\Gamma_h^\theta}u_h^\theta
   + \nu_h^\theta (\nu_h^\theta)^T (\nabla_{\Gamma_h^\theta} e_x^\theta)^T \nabla_{\Gamma_h^\theta}u_h^\theta \bigr)\Bigr) \d\theta
\\
&\le c \| e_u \|_{L^2(\Gamma_h(\xs))} \Bigl( \| e_u \|_{L^2(\Gamma_h(\xs))} 
\\
&\quad + \| \nabla_{\Gamma_h(\xs))} e_u \|_{L^2(\Gamma_h(\xs))} + 
\| \nabla_{\Gamma_h(\xs))} e_x \|_{L^2(\Gamma_h(\xs))}\,
\| \nabla_{\Gamma_h(\xs))} u_h^* \|_{L^\infty(\Gamma_h(\xs))}
\\
&\quad + \| \nabla_{\Gamma_h(\xs))} e_x \|_{L^\infty(\Gamma_h(\xs))} \,
\| \nabla_{\Gamma_h(\xs))} e_u \|_{L^2(\Gamma_h(\xs))} \Bigr)
\\
&\le C \normM{\eu} \Bigl( \normM{\eu} + \normA{\eu} + \normA{\ex} + \normA{\eu} \Bigr)
\\
&\le C  \normM{\eu} \Bigl( \normM{\eu} + \normA{\eu} + \normK{\ex} \Bigr).
\end{align*}

Combined, the above estimates  yield the following inequality:
\begin{align*}
    \half \diff \normM{\eu}^2 &+ \normA{\eu}^2 \leq \ C\,\normM{\eu}^2\\
    &\ +C \normM{\eu} \Bigl( \normK{\ev} + \normK{\ex} \Bigr)\\
    &\ + C \normM{\eu} \normK{\ev} + c \normM{\eu}  \normK{\ex} \\
    &\ + C \normM{\eu} \normK{\ev} \\
    &\ - \half \diff \Big( \test^T \big( \bfM(\bfx)-\bfM(\xs) \big) \eu\Big) \\
    &\ + C \normA{\eu} \normK{\ex} \\
    &\ + C h^{(\kappa-1)/2} \normA{\eu}^2 \\
    &\ + C \normM{\eu} \Bigl( \normM{\eu} + \normA{\eu} + \normK{\ex} \Bigr) \\
    &\ + C \normA{\eu} \|\du\|_{\star,\xs} .
\end{align*}
Estimating further, using Young's inequality and absorptions into $\normA{\eu}^2$ (using $h\leq h_0$ for a sufficiently small $h_0$), we obtain the following estimate,  where we can choose $\rho>0$  small at the expense of enlarging the constant in front of $ \normM{\eu}^2$:
\begin{equation}
\label{eq: PDE diff bound}
    \begin{aligned}
        \half \diff \normM{\eu}^2  + \half \normA{\eu}^2 \leq &\ c \normM{\eu}^2 + c \normK{\ex}^2
        + \rho\normK{\ev}^2  \\
        &\- \half \diff \Big( \test^T \big( \bfM(\bfx)-\bfM(\xs) \big) \eu\Big) + c \|\du\|_{\star,\xs}^2 .
    \end{aligned}
\end{equation}
(B) {\it Estimates in the surface equation:} Based on Section~\ref{section: uncoupled problem}, we obtain
\begin{align*}
    \normK{\ev}^2 \leq &\ c \normK{\ex}^2  + |\ev^T \big(\bfg(\bfx,\bfu) - \bfg(\xs,\us)\big)| + c \|\dv\|_{\star,\xs}^2 ,
\end{align*}
where the coupling term can be estimated based on (iii) in the proof of Proposition~\ref{proposition: stability - uncoupled problem} and (vii) above:
\bcl
$$
|\ev^T \big(\bfg(\bfx,\bfu) - \bfg(\xs,\us)\big)| \le \normM{\ev} \bigl( \normM{\eu} + \normA{\eu} + \normK{\ex} \bigr).
$$
We then obtain
\begin{equation}
\label{eq: surf dot bound}
    \normK{\ev}^2 \leq C \bigl( \normK{\ex}^2 +  \normM{\eu}^2
    + \normA{\eu}^2
    +  \|\dv\|_{\star,\xs}^2 \bigr) .
\end{equation}
As in Section~\ref{section: uncoupled problem}, this provides the estimate
\begin{equation}
\label{eq: surf diff bound}
    \half \diff \normK{\ex}^2 \leq C \bigl( \normK{\ex}^2 +  \normM{\eu}^2
    + \normA{\eu}^2
    +  \|\dv\|_{\star,\xs}^2 \bigr) .
\end{equation}
(C) {\it Combination:} We first insert \eqref{eq: surf dot bound} into \eqref{eq: PDE diff bound}, where we can choose $\rho>0$ so small that $C\rho\le 1/2$ for the constant $C$ in \eqref{eq: surf dot bound}.
Then we take a linear combination of  \eqref{eq: PDE diff bound} and \eqref{eq: surf diff bound} to obtain,  for a sufficiently small $\sigma>0$,
\begin{equation*}
    \begin{aligned}
        \diff \normM{\eu}^2  + &\, \frac12\normA{\eu}^2 +  \sigma  \diff \normK{\ex}^2 \\
        \leq &\ c \normM{\eu}^2  + c \normK{\ex}^2 + \diff \Big( \test^T \big( \bfM(\bfx)-\bfM(\xs) \big) \eu\Big) \\
        &\ + c \|\du\|_{\star,\xs}^2 + c \|\dv\|_{\star,\xs}^2 .
    \end{aligned}
\end{equation*}
We integrate both sides over $[0,t]$, for $0\leq t \leq t^*$, to get
\begin{equation*}
    \begin{aligned}
        & \normM{\eu(t)}^2  + \frac12 \int_0^t \normA{\eu(s)}^2 \d s + \sigma\normK{\ex(t)}^2 \\
        & \leq \normM{\eu(0)}^2  + \normK{\ex(0)}^2 + c \int_0^t \Bigl(\normM{\eu(s)}^2  + \normK{\ex(s)}^2 \Bigr)\d s\\
        &\ \qquad - \test(t)^T \big( \bfM(\bfx)-\bfM(\xs) \big) \eu(t) \\
        &\ \qquad + c \int_0^t \Bigl(\|\du(s)\|_{\star,\xs}^2 + \|\dv(s)\|_{\star,\xs}^2 \Bigr) \d s .
    \end{aligned}
\end{equation*}
The middle term can be further bounded using Lemmas~\ref{lemma: matrix differences}--\ref{lemma:theta-independence} and an $L^2-L^\infty-L^2$ estimate, as
\begin{align*}
    \test(t)^T \big( \bfM(\bfx)-\bfM(\xs) \big) \eu(t) = &\ \int_0^1 \int_{\Gamma_h^\theta} e_u^\theta \cdot (\nb_{\Gamma_h^\theta} \cdot e_x^\theta) e_u^\theta \,\d\theta \\
    \leq &\ c \normM{\eu(t)}^2 \|\nb_{\Ga_h(\xs)} \cdot e_x\|_{L^\infty(\Ga_h(\xs))} \\
    \leq &\ C h^{(\kappa-1)/2} \normM{\eu(t)}^2  ,
\end{align*}
where  we used the first bound from \eqref{eq: assumed bounds - 2} in the last inequality.

Absorbing this to the left-hand side and using Gronwall's inequality yields the stability estimate

    \begin{align}\nonumber
        \normM{\eu(t)}^2 + &\ \int_0^t \normA{\eu(s)}^2 \d s + \normK{\ex(t)}^2 \\
        \leq &\ c \int_0^t \Bigl(\|\du(s)\|_{\star,\xs}^2 + \|\dv(s)\|_{\star,\xs}^2  \Bigr)\d s .
        \label{eu-ex-est}
    \end{align}

Inserting this bound in \eqref{eq: surf dot bound}, squaring and integrating from $0$ to $t$ yields
$$
\int_0^t \normK{\ev(s)}^2\,\d s \le c \int_0^t\Bigl(\|\du(s)\|_{\star,\xs}^2 + \|\dv(s)\|_{\star,\xs}^2  \Bigr)\d s .
$$
With the assumed bounds of the defects, we obtain $O(h^k)$ error estimates for $0 \leq t \leq t^*$.
Finally, to show that $t^*=T$, we use the same argument as at the end of the  proof of Proposition~\ref{proposition: stability - uncoupled problem}.
\QED \end{proof}

\bcl
\begin{remark} If the coupling function $g=g(u)$ in \eqref{eq: coupled problem} does not depend on the tangential gradient of $u$, then the term $\normA{\eu}^2$ does not appear in the bound \eqref{eq: surf dot bound}. Therefore, inserting the  estimate \eqref{eu-ex-est} into \eqref{eq: surf dot bound} then yields a pointwise stability estimate for $\ev$: uniformly for $0\le t \le T$,
$$
\normK{\ev(t)}^2 \le C \|\dv(t)\|_{\star,\xs}^2 +  C\int_0^t \Bigl(\|\du(s)\|_{\star,\xs}^2 + \|\dv(s)\|_{\star,\xs}^2  \Bigr)\d s.
$$
\end{remark}

\ecl

\section{Geometric estimates}
\label{section: geometric estimates}

In this section we give further notations and some technical lemmas from \cite{highorder} that will be used later on. Most of the results are high-order and time-dependent extensions of geometric approximation  estimates shown in \cite{Dziuk88,DziukElliott_ESFEM,DziukElliott_L2} and \cite{Demlow2009}.

\subsection{The interpolating surface}
\label{subsection: interpolating surface}
We return to the setting of Section~\ref{section: problem}, where $X(\cdot,t)$ defines a smooth surface $\Gamma(t)=\Gamma(X(\cdot,t))$.  For an admissible triangulation of $\Gamma(t)$ with nodes $x_j^*(t)=X(p_j,t)$ and the corresponding nodal vector $\bfx^*(t)=(x_j^*(t))$, we define the interpolating surface by
$$
X_h^*(p_h,t) = \sum_{j=1}^N x_j^*(t) \, \phi_j[\bfx(0)](p_h), \qquad p_h \in \Gamma_h^0,
$$
which has the properties that $X_h^*(p_j,t)=x_j^*(t)=X(p_j,t)$ for $j=1,\dots,N$,  and
$$
\Gamma_h^*(t):=\Gamma_h(\bfx^*(t))=\Gamma(X_h^*(\cdot,t)).
$$
In the following we drop the argument $t$ when it is not essential.
The velocity of the interpolating surface $\Ga_h^*$, defined as in Section~\ref{section:ESFEM}, is denoted by $v_h^*$.

\subsection{Approximation results}
\label{section: lift}

\newcommand{\pr}{\textnormal{Pr}}
\newcommand{\wein}{\mathcal{H}}
The lift of a function $\eta_h:\Gamma_h^*(t)\to\R$ is again denoted by $\eta_h^l:\Gamma(t)\to\R$, defined via the oriented distance function $d$ between $\Gamma_h^*(t)$ and $\Gamma(t)$ provided that the surfaces are sufficiently close (which is the case if $h$ is sufficiently small).

\begin{lemma}{\bf (Equivalence of norms \cite[Lemma~3]{Dziuk88}, \cite{Demlow2009})}
\label{lemma: equivalence of norms}
    Let $\eta_h : \Ga_h^*(t) \to \R$ with lift $\eta_h^l : \Ga(t) \to \R$. Then the 
    \bbk $L^p$ and $W^{1,p}$ norms on the discrete and continuous surfaces \ebk 
    are equivalent for $1\le p \le \infty$, uniformly in the mesh size $h\le h_0$ (with sufficiently small $h_0>0$) and in $t\in [0,T]$.
\end{lemma}
In particular, there is a constant $c$ such that for $h\le h_0$ and $0\le t \le T$,
\begin{align*}
    c\inv \|\eta_h\|_{L^2(\Ga_h^*(t))} \leq &\ \|\eta_h^l\|_{L^2(\Ga(t))} \leq c\|\eta_h\|_{L^2(\Ga_h^*(t))} , \\
    c\inv \|\eta_h\|_{H^1(\Ga_h^*(t))} \leq &\ \|\eta_h^l\|_{H^1(\Ga(t))} \leq c\|\eta_h\|_{H^1(\Ga_h^*(t))} .
\end{align*}

%

Later on the following estimates will be used. They have been shown in \cite{highorder}, based on \cite{Demlow2009} and \cite{DziukElliott_L2}.
\begin{lemma}
\label{lemma: geometric est}
   \bbk 
   Let $\Ga(t)$ and $\Ga_h^*(t)$ be as above in Section~\ref{subsection: interpolating surface}. 
   Then, for $h\leq h_0$ with a sufficiently small $h_0>0$, we have the following estimates for the distance function $d$ from \eqref{eq: distance function}, and for the error in the normal vector: 
   \ebk
    $$
        \|d\|_{L^\infty(\Ga_h^*(t))} \leq c h^{k+1}, \qquad
        \|\nu_{\Ga(t)} - \nu_{\Ga_h^*(t)}^l\|_{L^\infty(\Ga(t))} \leq c h^{k} ,
    $$
    with constants independent of $h\le h_0$ and $t\in [0,T]$.
\end{lemma}

\subsection{Bilinear forms and their estimates}
\label{section: bilinear forms}

We use surface-dependent bilinear forms defined similarly as in \cite{DziukElliott_L2}: Let $X$ be a given surface with velocity $v$, with interpolation surface $\Xs$ with velocity $v_h^*$. For arbitrary $z,\vphi \in H^1(\Ga(X))$ and for their discrete analogs $Z_h, \phi_h \in S_h(\xs)$:
\begin{equation*}
    \begin{aligned}[c]
        m(X;z,\vphi)                &= \int_{\Ga(X)}\!\!\!\! z \vphi, \\
        a(X;z,\vphi)                &= \int_{\Ga(X)}\!\!\!\! \nbg z \cdot \nbg \vphi, \\
        q(X;v;z,\vphi)              &= \int_{\Ga(X)}\!\!\!\! (\nbg \cdot v) z\vphi, \\
    \end{aligned}
    \quad
    \begin{aligned}[c]
        m(\Xs;Z_h,\phi_h)         &= \int_{\Ga(\Xs)}\!\!\!\! Z_h \phi_h, \\
        a(\Xs;Z_h,\phi_h)         &= \int_{\Ga(\Xs)}\!\!\!\! \nbgh Z_h \cdot \nbgh \phi_h, \\
        q(\Xs;v_h^*;Z_h,\phi_h)     &= \int_{\Ga(\Xs)}\!\!\!\! (\nbgh \cdot v_h^*) Z_h \phi_h, \\
    \end{aligned}
\end{equation*}
where the discrete tangential gradients are understood in a piecewise sense.
For more details see \cite[Lemma 2.1]{DziukElliott_L2} (and the references in the proof), or \cite[Lemma 5.2]{DziukElliott_acta}.
%

%
%

We start by defining a discrete velocity on the smooth surface, denoted by $\vhat$. 
\bbk We follow Section~5.3 of \cite{highorder}, where the high-order ESFEM generalization of the discrete velocity on $\Ga(X)$ from Sections~4.3 and 5.3 of \cite{DziukElliott_L2} is discussed.
\ebk
Using the lifted elements, $\Ga(X)$ is decomposed into curved elements whose \bbk Lagrange points \ebk move with the velocity $\vhat$ defined by
\begin{equation*}
    \vhat \big((\Xs)^l(\cdot,t),t\big) = \diff (\Xs)^l(\cdot,t).
\end{equation*}

Discrete material derivatives on $\Ga(\Xs)$ and $\Ga(X)$ are given by
\begin{equation*}
    \begin{aligned}
        \mat_{\Vs} \vphi_h =&\ \pa_t \vphi_h + \Vs \cdot \nb \vphi_h , \\
        \mat_{\vhat} \vphi_h^l =&\ \pa_t \vphi_h^l + \vhat \cdot \nb \vphi_h^l ,
    \end{aligned}
    \qquad\qquad (\vphi_h \in S_h(\xs)) .
\end{equation*}
In \cite[Lemma~4.1]{DziukElliott_L2} it was shown that the transport property of the basis functions carries over to the lifted basis functions $\phi_j[\xs]$:
\begin{equation*}
    \mat_{\vhat} \phi_j[\xs]^l = (\mat_{\Vs} \phi_j[\xs])^l = 0 , \qquad
    (j=1,\dotsc,N) .
\end{equation*}
Therefore, the above discrete material derivatives and the lift operator satisfy, for $\vphi_h \in S_h(\Xs)$,
\begin{equation}
\label{eq: lift and material derivatives}
    \mat_{\vhat} \vphi_h^l = (\mat_{\Vs} \vphi_h)^l .
\end{equation}

\begin{lemma}{\bf (Transport properties \cite[Lemma~4.2]{DziukElliott_L2})}
\label{lemma: transport prop}
    For any $z(t),\vphi(t) \in H^1(\Ga(X(\cdot,t)))$,
    \begin{align*}
        \diff m(X;z,\vphi) =&\ m(X;\mat z,\vphi) + m(X;z,\mat \vphi) + q(X;v;z,\vphi) .
    \end{align*}
    \bbk 
    The same formulas hold when $\Ga(X)$ is considered as the lift of the discrete surface $\Ga(X_h^*)$ (i.e. $\Ga(X)$ can be decomposed into curved elements which are lifts of the elements of $\Ga(X_h^*)$), moving with the velocity $\vhat$:
    \ebk
    \begin{align*}
        \diff m(X;z,\vphi) =&\ m(X;\mat_{\vhat} z,\vphi) + m(X;z,\mat_{\vhat} \vphi) + q(X;\vhat;z,\vphi) .
    \end{align*}
    Similarly, in the discrete case, for arbitrary $z_h(t), \vphi_h(t), \bbk \mat_{\Vs} z_h(t), \mat_{\Vs} \vphi_h(t) \ebk \in S_h(\xs(t))$ we have:
    \begin{align*}
        \diff m(\Xs;z_h,\vphi_h) =&\ m(\Xs;\mat_{\Vs} z_h,\vphi_h) + m(\Xs;z_h,\mat_{\Vs} \vphi_h) + q(\Xs;\Vs;z_h,\vphi_h) ,
    \end{align*}
    where $\Vs$ is the velocity of the surface $\Ga(\Xs)$.

\end{lemma}


The following estimates, proved in Lemma~5.6 of \cite{highorder}, will play a crucial role in the defect bounds later on.
\begin{lemma}[Geometric perturbation errors]
\label{lemma: geometric perturbation errors}
    For any $Z_h,\psi_h \in S_h(\xs)$ where $\Ga(\Xs)$ is the interpolation surface of piecewise polynomial degree $k$,  we have the following bounds, \bbk for $h\leq h_0$ with a sufficiently small $h_0>0$, \ebk
    \begin{align*}
        \big| m(X; Z_h^l,\vphi_h^l) - m(\Xs; Z_h,\vphi_h) \big| \leq&\ c h^{k+1} \|Z_h^l\|_{L^2(\Ga(X))} \|\vphi_h^l\|_{L^2(\Ga(X))}, \\
        \big| a(X; Z_h^l,\vphi_h^l) - a(\Xs; Z_h,\vphi_h) \big| \leq&\ c h^{k+1} \|\nbg Z_h^l\|_{L^2(\Ga(X))} \|\nbg \vphi_h^l\|_{L^2(\Ga(X))} , \\
        \big| q(X;\vhat; Z_h^l,\vphi_h^l) - q(\Xs;\Vs; Z_h,\vphi_h) \big| \leq&\ c h^{k+1} \|Z_h^l\|_{L^2(\Ga(X))} \|\vphi_h^l\|_{L^2(\Ga(X))} .
    \end{align*}
    The constant $c$ is independent of $h$ and $t\in[0,T]$.
%
\end{lemma}

\subsection{Interpolation error estimates for evolving surface finite element  functions}

For any $u\in H^{k+1}(\Ga(X))$, there is a unique piecewise polynomial surface finite element interpolation of degree~$k$ in the nodes $x_j^*$, denoted by $\widetilde I_h u \in S_h(\xs)$. We set $ I_h u :=(\widetilde I_h u)^l   : \Ga(X) \to \R$. Error estimates for this interpolation are obtained from \cite[Proposition 2.7]{Demlow2009} by carefully studying the time dependence of the constants, cf.\ \cite{highorder}.

\newcommand{\Ih}{\widetilde{I}_h}
\begin{lemma}
\label{lemma: interpolation error}
    There exists a constant $c>0$ independent of \bbk $h \leq h_0$, with a sufficiently small $h_0>0$, \ebk and $t$ such that for $u(\cdot,t)\in H^{k+1}(\Ga(t))$, for $0\leq t \leq T$, 
    \begin{align*}
        \|u - I_{h} u \|_{L^{2}(\Ga(X))} + h \| \nbg( u - I_{h} u) \|_{L^{2}(\Ga(X))} \leq&\ c h^{k+1} \|u\|_{H^{k+1}(\Ga(X))} .
    \end{align*}
\end{lemma}
The same result holds for vector valued functions. As it will always be clear from the context we do not distinguish between interpolations for scalar and vector valued functions.

\section{Defect bounds}
\label{section: defect bounds}

In this section we show that the assumed defect estimates of Proposition~\ref{proposition: stability - uncoupled problem} and \ref{proposition: stability - coupled problem} are indeed fulfilled when the projection $\Pi_h$ is chosen to be the piecewise $k$th-degree polynomial interpolation operator $I_h$ for $k\geq2$.

The interpolations satisfy the discrete problem \eqref{uh-vh-equation}--\eqref{xh} only up to some defects. These defects are denoted by $d_u \in S_h(\xs), d_v \in S_h(\xs)^3$, with $\xs(t)$  the vector of exact nodal values $x_j^*(t)=X(p_j,t)\in\Gamma(t)$, and are given as follows: for all $\vphi_h \in S_h(\xs)$ with $\mat_{\Vs} \vphi_h =0$ and $\psi_h \in S_h(\xs)^3$,
\begin{equation*}
    \begin{aligned}
        \int_{\Ga_h(\xs)}\!\!\!  d_u \vphi_h =&\ \diff \int_{\Ga_h(\xs)} \!\! \Ih u\, \vphi_h
        + \int_{\Ga_h(\xs)}\!\!\! \nabla_{\Ga_h(\xs)} \Ih u \cdot \nabla_{\Ga_h(\xs)}  \vphi_h \\
        - &\ \int_{\Ga_h(\xs)} f(\Ih u,\nabla_{\Ga_h(\xs)} \Ih u)\,\varphi_h , \\
        \int_{\Ga_h(\xs)} \!\!\! d_v \cdot \psi_h =&\ \int_{\Ga_h(\xs)}\!\! \Ih v \cdot \psi_h
        + \alpha \int_{\Ga_h(\xs)} \!\!\! \nabla_{\Ga_h(\xs)}  \Ih v \cdot \nabla_{\Ga_h(\xs)}  \psi_h \\
        - &\ \int_{\Ga_h(\xs)}  g(\Ih u,\nabla_{\Ga_h(\xs)} \Ih u) \,\nu_{\Ga_h(\xs)} \cdot \psi_h .
    \end{aligned}
\end{equation*}
Later on the vectors of nodal values of the defects $d_u$ and $d_v$ are denoted by $\du \in \R^N$ and $\dv\in\R^{3N}$, respectively. These vectors 
satisfy \eqref{eq: DAE form - coupled - solution}.

\begin{lemma}
\label{lemma: semidiscrete residual}
    Let the solution $u$, the surface $X$ and its velocity $v$ be all sufficiently smooth. Then there exists a constant $c>0$ such that for \bbk all $h\leq h_0$, with a sufficiently small $h_0>0$, \ebk and for all $t\in[0,T]$, the  defects $d_u$ and $d_v$ of the $k$th-degree finite element interpolation
    are bounded as
    \begin{align*}
        \|\du\|_{\star,\xs}=&\ \|d_u\|_{H_h\inv(\Ga(\Xs))} \leq c h^k , \\
        \|\dv\|_{\star,\xs}=&\ \|d_v\|_{H_h\inv(\Ga(\Xs))} \leq c h^k ,
    \end{align*}
    where the $H_h\inv$-norm is defined in \eqref{eq: star and H^-1 norm equality}. The constant $c$ is independent of $h$ and $t\in[0,T]$.
\end{lemma}
\begin{proof}
(i) We start from an identity for the dual norm as in \eqref{eq: star and H^-1 norm equality}, (omitting the  argument $\xs$ of the matrices):
\begin{equation*}
    \begin{aligned}
        \|\du\|_{\star,\xs} =&\ (\du^T \bfM(\bfM + \bfA)\inv \bfM\du )^\half
        =\|d_u\|_{H_h\inv(\Ga(\Xs))}  .
    \end{aligned}
\end{equation*}

In order to estimate the defect in $u$, we subtract \eqref{eq: coupled problem - weak form} from the above equation, and perform almost the same proof as in \cite[Section~7]{DziukElliott_L2}.
We use the bilinear forms and the discrete versions of the transport properties from Lemma~\ref{lemma: transport prop}. We obtain, for any $\vphi_h \in S_h(\xs)$ with $\mat_{\Vs} \vphi_h =0$,
\begin{align*}
    m(\Xs ; d_u,\vphi_h) =&\ \diff m(\Xs ; \Ih u,\vphi_h) + a(\Xs ; \Ih u,\vphi_h) \\
    &\ - m(\Xs ; f(\Ih u,\nbgh \Ih u),\vphi_h) \\
    =&\  m(\Xs ; \mat_{\Vs} \Ih u,\vphi_h) + q(\Xs ;\Vs; \Ih u, \vphi_h) + a(\Xs ; \Ih u,\vphi_h) \\
    &\ - m(\Xs ; f(\Ih u,\nbgh \Ih u),\vphi_h) ,
\end{align*}
and
\begin{align*}
    0 =&\ \diff m(X ; u,\vphi_h^l) + a(X ; u,\vphi_h^l) - m(X ; f(u,\nabla_{\Ga(X)} u),\vphi_h^l) \\
    =&\ m(X ; \mat_{\vhat} u,\vphi_h^l) + q(X;\vhat;u,\vphi_h^l) + a(X ; u,\vphi_h^l) - m(X ; f(u,\nabla_{\Ga(X)} u),\vphi_h^l).
\end{align*}
Subtracting the two equation yields
\begin{equation*}
    \begin{aligned}
        m(\Xs ; d_u,\vphi_h)
        =&\ m(\Xs ; \mat_{\Vs} \Ih u, \vphi_h) - m(X;\mat_{\vhat} u,\vphi_h^l) \\
        &\ + q(\Xs ;\Vs; \Ih u, \vphi_h) - q(X;\vhat;u,\vphi_h^l) \\
        &\ + a(\Xs ; \Ih u, \vphi_h) - a(X;u,\vphi_h^l) \\
        &\ - \Big(m(\Xs ; f(\Ih u,\nbgh \Ih u), \vphi_h) - m(X;f(u,\nbg u),\vphi_h^l)\Big) .
    \end{aligned}
\end{equation*}
We bound all the terms pairwise, by using the interpolation estimates of Lemma~\ref{lemma: interpolation error} and the estimates for the geometric perturbation errors of the bilinear forms of Lemma~\ref{lemma: geometric perturbation errors}. For the first pair, using that $(\mat_{\Vs} \Ih u)^l = \mat_{\vhat} I_h u$, we obtain
\begin{align*}
        \big|m(\Xs ; \mat_{\Vs} \Ih u, \vphi_h) - m(X;\mat_{\vhat} u,\vphi_h^l)\big| \leq &\
        \big|m(\Xs ; \mat_{\Vs} \Ih u, \vphi_h) - m(X;\mat_{\vhat} I_h u,\vphi_h^l)\big| \\
        &\ + \big|m(X;I_h \mat_{\vhat} u - \mat_{\vhat} u,\vphi_h^l)\big| \\
        \leq &\ c h^{k+1} \|\vphi_h^l\|_{L^2(\Ga(X))} .
\end{align*}
For the second pair we obtain
\begin{align*}
    \big|q(\Xs ;\Vs; \Ih u, \vphi_h) - q(X;\vhat;u,\vphi_h^l)\big| \leq &\
    \big|q(\Xs ;\Vs; \Ih u, \vphi_h) - q(X;\vhat;I_h u,\vphi_h^l)\big| \\
    &\ + \big|q(X ;\Vs; I_h u - u, \vphi_h)\big| \\
    \leq &\ c h^{k+1} \|\vphi_h^l\|_{L^2(\Ga(X))} .
\end{align*}
The third pair is estimated by
\begin{align*}
    \big|a(\Xs ; \Ih u, \vphi_h) - a(X;u,\vphi_h^l) \big| \leq &\
    \big| a(\Xs ; \Ih u, \vphi_h) - a(X; I_h u,\vphi_h^l) \big| \\
    &\ + \big| a(X; I_h u - u,\vphi_h^l)\big| \\
    \leq &\ c h^k \|\nbg \vphi_h^l\|_{L^2(\Ga(X))} .
\end{align*}
For the last pair we use the fact that $(f(u,\nbg u))^{-l}=f(u^{-l},(\nbg u)^{-l})$ and the local Lipschitz continuity of the function $f$, to obtain
\begin{align*}
    &\  \big|m(\Xs ; f(\Ih u,\nbgh \Ih u) , \vphi_h) - m(X;f(u,\nbg u) , \vphi_h^l)\big|\\
    \leq &\ \big| m(\Xs ; f(\Ih u,\nbgh \Ih u) - f(u^{-l},(\nbg u)^{-l}) , \vphi_h) \big|\\
    &\ + \big| m(\Xs ; f(u,\nbg u)^{-l} , \vphi_h) - m(X;f(u,\nbg u) , \vphi_h^l)\big| \\
    \leq &\ c \|f(\Ih u,\nbgh \Ih u) - f(u^{-l},(\nbg u)^{-l})\|_{L^2(\Ga(\Xs))} \|\vphi_h^l\|_{L^2(\Ga(X))} \\
    &\ + ch^{k+1} \|\vphi_h^l\|_{L^2(\Ga(X))} .
\end{align*}
The first term is estimated, using the local Lipschitz continuity of $f$ and equivalence of norms, by
\begin{align*}
    &\ \|f(\Ih u,\nbgh \Ih u) - f(u^{-l},(\nbg u)^{-l})\|_{L^2(\Ga(\Xs))} \\
    \leq &\ \|f\|_{W^{1,\infty}} \Big( c\|I_h u - u\|_{L^2(\Ga(X))} + c\|\nbg (I_h u - u)\|_{L^2(\Ga(X))} \\ &\ +  c\|(\nbgh u^{-l})^l - \nbg u\|_{L^2(\Ga(X))}\Big) ,
\end{align*}
where the first two terms are bounded by $O(h^k)$ using interpolation estimates, while the third term is bounded, using Remark~4.1 in \cite{DziukElliott_L2} and Lemma~\ref{lemma: geometric est}, as
\begin{align*}
    \|(\nbgh u^{-l})^l - \nbg u\|_{L^2(\Ga(X))}
    \leq &\ ch^{k}.
\end{align*}

Thus for the fourth pair we obtained
\begin{equation*}
    \big|m(\Xs ; f(\Ih u,\nbgh \Ih u) , \vphi_h) - m(X;f(u,\nbg u) , \vphi_h^l)\big| \leq ch^{k} \|\vphi_h^l\|_{L^2(\Ga(X))} .
\end{equation*}

Altogether, we have
\begin{equation*}
    m(\Xs ; d_u,\vphi_h) \leq c h^k \|\vphi_h^l\|_{H^1(\Ga(X))} ,
\end{equation*}
which, by the equivalence of norms given by Lemma \ref{lemma: equivalence of norms}, shows the first bound of the stated lemma.


(ii) In order to estimate the defect in $v$, similarly as previously we subtract \eqref{eq: coupled problem - weak form} from the above equation and use the bilinear forms to obtain
\begin{align*}
    &\ m(\Xs ; d_\bfv,\psi_h) \\
    &\ = m(\Xs ; \Ih v, \psi_h) - m(X;v,\psi_h^l) \\
    &\ \qquad + \alpha \Big(a(\Xs ; \Ih v, \psi_h) - a(X;v,\psi_h^l)\Big) \\
    &\ \qquad + m(\Xs ; g(\Ih u,\nbgh \Ih u) \nu_{\Ga(\Xs)},\psi_h) - m(X;g(u,\nbg u) \nu_{\Ga(X)}, \psi_h^l) .
\end{align*}
Similarly as in the previous part, these three pairs are bounded pairwise.  For the first pair we have
\begin{align*}
    |m(\Xs ; \Ih v, \psi_h) - m(X;v,\psi_h^l)
    |\leq &\ | m(\Xs ; \Ih v, \psi_h) - m(X;I_h v,\psi_h^l) | \\
    &\ +| m(X;I_h v - v ,\psi_h^l) | \\
    \leq &\ c h^{k+1} \|\psi_h^l\|_{L^2(\Ga(X))} .
\end{align*}
For the second pair we use the interpolation estimate to bound
\begin{align*}
    |a(\Xs ; \Ih v, \psi_h) - a(X;v,\psi_h^l) |\leq&\ | a(\Xs ; \Ih v, \psi_h) -  a(X;I_h v,\psi_h^l)| \\
    &\ +| a(X;I_h v - v,\psi_h^l) | \\
    \leq &\ c h^k \|\nbg \psi_h^l\|_{L^2(\Ga(X))} .
\end{align*}
The third pair we estimate, similarly to the nonlinear pair above, by
\begin{align*}
    &\  |m(\Xs ; g(\Ih u,\nbgh \Ih u) \nu_{\Ga(\Xs)},\psi_h) - m(X;g(u,\nbg u) \nu_{\Ga(X)}, \psi_h^l)|\\
    \leq &\ | m(\Xs ; (g(\Ih u,\nbgh \Ih u) - g(u,\nbg u)^{-l} ) \nu_{\Ga(\Xs)}, \psi_h)|\\
    &\ + | m(\Xs ; g(u,\nbg u)^{-l} ( \nu_{\Ga(\Xs)} -\nu_{\Ga(X)}^{-l} ), \psi_h)|\\
    &\ + | m(\Xs ; g(u,\nbg u)^{-l} \nu_{\Ga(X)}^{-l}, \psi_h) - m(X;g(u,\nbg u) \nu_{\Ga(X)}, \psi_h^l) | \\
    \leq &\ c h^{k} \|g\|_{W^{1,\infty}} \|\psi_h^l\|_{L^2(\Ga(X))} + c \|\nbg(X - \Xs)\|_{L^2(\Ga(X))} \|\psi_h^l\|_{L^2(\Ga(X))} \\
    &\ + ch^{k+1} \|g\|_{L^2} \|\psi_h^l\|_{L^2(\Ga(X))} \\
    \leq &\ c h^{k} \|g\|_{W^{1,\infty}} \|\psi_h^l\|_{L^2(\Ga(X))} + ch^k \|\psi_h^l\|_{L^2(\Ga(X))} \\
    \leq &\ c h^k \|\psi_h^l\|_{L^2(\Ga(X))} ,
\end{align*}
where we have used the local Lipschitz boundedness of the function $g$, the interpolation estimate, Lemma~\ref{lemma: geometric est}, and Lemma~\ref{lemma: geometric perturbation errors}, through a similar argument as above for the semilinear term with $f$.

Finally, the combination of these bounds yields
\begin{equation*}
    m(\Xs ; d_v,\psi_h) \leq c h^k \|\psi_h^l\|_{H^1(\Ga(X))} ,
\end{equation*}
providing the asserted bound on $\bfd_\bfv$.
\QED \end{proof}

\section{Proof of Theorem~\ref{theorem: coupled error estimate}}
\label{section: proof completed}
The errors are decomposed using interpolations and the definition of lifts from Section~\ref{section:lifts}: omitting the argument $t$,
\begin{align*}
    u_h^{L}  - u  =&\ \big(\widehat u_h  - \Ih u \big)^{l} + \big(I_h u  - u \big) , \\
    v_h^{L}  - v  =&\ \big(\widehat v_h  - \Ih v \big)^{l} + \big(I_h v  - v \big) , \\
   X_h^L  - X  =&\ \big( \widehat X_h  - \Ih X  \big)^{l} +  \big(I_h X  - X \big).
\end{align*}
The last terms  in these formulas can be bounded in the $H^1(\Gamma)$ norm by $Ch^k$, using the interpolation bounds of Lemma~\ref{lemma: interpolation error}.

To bound the first terms on the right-hand sides, we first use the defect bounds of Lemma~\ref{lemma: semidiscrete residual}, which then together with the stability estimate of Proposition~\ref{proposition: stability - coupled problem} proves the result,
since by the norm equivalences of Lemma~\ref{lemma: equivalence of norms} and equations \eqref{M-L2}--\eqref{A-H1} we have (again omitting the argument $t$)
\begin{align*}
    &\| \big(\widehat u_h - \Ih u\big)^{l} \|_{L^2(\Ga)} \le
    c \| \widehat u_h - \Ih u \|_{L^2(\Ga_h^*)}
    = c \normM{\eu},
    \\
    &\| \nabla_{\Gamma} \big(\widehat u_h - \Ih u\big)^{l} \|_{L^2(\Ga)} \le
    c \| \nabla_{\Gamma_h^*} \big(\widehat u_h - \Ih u\big) \|_{L^2(\Ga_h^*)}
    = c \normA{\eu},
\end{align*}
and similarly for $\widehat v_h - \Ih v$ and $\widehat X_h - \Ih X$.
%

\section{Extension to other velocity laws}
\label{section: extensions}

In this section we consider the extension of our results to different velocity laws: adding a mean curvature term to the regularized velocity law considered so far, and a dynamic velocity law. We concentrate on the velocity laws without coupling to the surface PDE, since the coupling can be dealt with in the same way as previously. We only consider the stability of the evolving surface finite element discretization, since bounds for the consistency error are obtained by the same arguments as before.

\subsection{Regularized mean curvature flow}

We next extend our results to the case where the velocity law contains a mean curvature term:
\begin{equation}
\label{eq:reg_mcf}
    v - \alpha \Delta_{\Ga(X)} v - \beta \Delta_{\Ga(X)} X = g(\cdot,t) \nu_{\Ga(X)} ,
\end{equation}
where $g:\R^3\times\R\to\R$ is a given Lipschitz continuous function of $(x,t)$, and $\alpha>0$ and $\beta>0$ are fixed parameters.
Here \(\Delta_{\Ga(X)} X\) is a suggestive notation for \( -  H\normal\), where \(H\)
denotes the mean curvature of  the surface $\Ga(X)$. (More
  precisely,  $\Delta_{\Ga(X)} \id = - H \nu_{\Ga(X)}$.)

The corresponding differential-algebraic system reads
\begin{equation}
\label{eq:fe_reg_mcf}
    \regmass(\nnodes) \bfv + \stiff(\nnodes) \nnodes = \bfg(\bfx) ,
\end{equation}
where $\regmass(\nnodes)$ is again defined by \eqref{eq: K matrix def}
and where we now write $\stiff(\nnodes)$ for the matrix $\beta I_3\otimes \stiff(\nnodes)$ with $\stiff(\nnodes)$ of Section~\ref{subsection:DAE}.

Similarly as before the corresponding error equation is given as
\begin{equation*}
    \begin{aligned}
        \bfK(\xs) \ev + \bfA(\xs) \ex
        =&\ -  \bigl(\bfK(\bfx) - \bfK(\xs)\bigr) \vs - \bigl(\bfK(\bfx) - \bfK(\xs)\bigr) \ev \\
        &\ - \bigl(\bfA(\bfx) - \bfA(\xs)\bigr) \xs - \bigl(\bfA(\bfx) - \bfA(\xs)\bigr) \ex  \\
        &\ + \big(\bfg(\bfx) - \bfg(\xs)\big) - \bfM(\xs)\dv 
    \end{aligned}
\end{equation*}
\bcl together with $\dot\bfe_\bfx = \ev$. 

\begin{proposition}
\label{proposition: stability - with MC term}
    Under the assumptions of Proposition~\ref{proposition: stability - uncoupled problem},
  there exists  $h_0>0$ such that the following stability estimate holds for all $h\leq h_0$, for $0\le t \le T$:
    \begin{align*}
        & \|\ex(t)\|_{\bfK(\xs(t))}^2 \leq C \int_0^t  \|\dv(s)\|_{\star,\xs}^2 \,\d s,\\
        &  \|\ev(t)\|_{\bfK(\xs(t))}^2 \le C \|\dv(t)\|_{\star,\xs}^2
        + C \int_0^t  \|\dv(s)\|_{\star,\xs}^2 \,\d s.
    \end{align*}
    The constant $C$ is independent of $t$ and $h$, but depends on the final time $T$, and on the  parameters $\alpha$ and $\beta$. 
\end{proposition}
\ecl
\begin{proof}
We detail only those parts of the proof of Proposition~\ref{proposition: stability - uncoupled problem} where the mean curvature term introduces differences, otherwise exactly the same proof applies.

In order to prove the stability estimate we again test with $\ev$, and obtain
\begin{equation*}
    \begin{aligned}
        \normK{\ev}^2  =&\ -  \ev^T \bigl(\bfK(\bfx) - \bfK(\xs)\bigr) \vs - \ev^T \bigl(\bfK(\bfx) - \bfK(\xs)\bigr) \ev \\
        &\ - \ev^T \bigl(\bfA(\bfx) - \bfA(\xs)\bigr) \xs - \ev^T \bigl(\bfA(\bfx) - \bfA(\xs)\bigr) \ex - \ev^T \bfA(\xs) \ex \\
        &\ + \ev^T \big(\bfg(\bfx) - \bfg(\xs)\big) - \ev^T \bfM(\xs)\dv .
    \end{aligned}
\end{equation*}

Every term is estimated exactly as previously in the proof of Proposition~\ref{proposition: stability - uncoupled problem}, except the terms corresponding to the mean curvature term, involving the stiffness matrix $\bfA$. They are estimated by the same techniques as previously:
\begin{align*}
    \ev^T \bigl(\bfA(\bfx) - \bfA(\xs)\bigr) \xs + \ev^T \bigl(\bfA(\bfx) - \bfA(\xs)\bigr) \ex
    \leq &\ \frac{1}{6} \normK{\ev}^2 + c \normK{\ex}^2 ,\\
    \ev^T \bfA(\xs) \ex \leq &\ \frac{1}{6} \normK{\ev}^2 + c \normK{\ex}^2 .
\end{align*}

Altogether we obtain the error bound
\begin{equation*}
        \normK{\ev}^2 \leq c \normK{\ex}^2 + c \|\dv\|_{\star,\xs}^2 ,
\end{equation*}
which is exactly \eqref{eq: velocity stability bound}. The proof is then completed as before.
\QED \end{proof}

With Proposition \ref{proposition: stability - with MC term} and the appropriate defect bounds,
Theorem~\ref{theorem: coupled error estimate} extends directly to the system with mean curvature term in the regularized velocity law.

\subsection{A dynamic velocity law}

Let us consider the dynamic velocity law, again without coupling to a surface PDE:
\begin{equation*}
    \mat v + v \nb_{\Ga(X)} \cdot v  - \alpha \laplace_{\Ga(X)} v = g(\cdot,t) \, \nu_{\Ga(X)},
\end{equation*}
where again $g:\R^3\times\R\to\R$ is a given Lipschitz continuous function of $(x,t)$, and $\alpha>0$ is a fixed parameter. This problem is considered together with the ordinary differential equations \eqref{velocity} for the positions $X$ determining the surface $\Gamma(X)$. Initial values are specified for~$X$ and $v$.

The weak formulation and the semidiscrete problem can be obtained by a similar argument as for the PDE on the surface in Section~\ref{section: coupled problem}.
%
Therefore we immediately present the ODE formulation of the semidiscretization. As in Section~\ref{subsection:DAE}, the nodal vectors $\bfv\in\R^{3N}$ of the finite element function $v_h$, together with the surface nodal vector $\bfx\in\R^{3N}$ satisfy a system of ODEs with matrices and driving term as in Section~\ref{section: uncoupled problem}:
\begin{equation}
\label{eq: DAE form - dynamic uncoupled}
    \begin{aligned}
        \diff \Big(\bfM(\bfx) \bfv\Big) + \bfA(\bfx)\bfv =&\ \bfg(\bfx,t) , \\
        \dot\bfx =&\ \bfv .
    \end{aligned}
\end{equation}

By using the same notations for the exact positions $\xs(t)\in\R^{3N}$, for the \bbk interpolated \ebk exact velocity $\vs(t)\in\R^{3N}$, and for the defect $\dv(t)$, 
we obtain that they fulfill the following equation
\begin{equation*}
    \begin{aligned}
        \diff \Big(\bfM(\xs) \vs\Big) + \bfA(\xs)\vs =&\ \bfg(\xs,t) + \bfM(\xs)\dv , \\
        {\dot \bfx}^* =&\ \bcl \vs  \ecl.
    \end{aligned}
\end{equation*}
By subtracting this from \eqref{eq: DAE form - dynamic uncoupled}, and using similar arguments as before, we obtain the error equations for the surface nodes and velocity:
\begin{equation*}
    \begin{aligned}
        \diff \Big(\bfM(\xs) \ev\Big) + \bfA(\xs)\ev =&\ -\diff \Big(\big(\bfM(\bfx) - \bfM(\xs)\big) \vs \Big) \\
        &\ -\diff \Big(\big(\bfM(\bfx) - \bfM(\xs)\big) \ev \Big)\\
        &\ - \big(\bfA(\bfx) - \bfA(\xs)\big) \vs  \\
        &\ - \big(\bfA(\bfx) - \bfA(\xs)\big) \ev \\
        &\ + \big(\bfg(\bfx)-\bfg(\xs)\big) - \bfM(\xs)\dv \\
        \dotex =&\ \bcl \ev \ecl .
    \end{aligned}
\end{equation*}


\bcl
We then have the following stability result.

\begin{proposition}
\label{proposition: stability - dynamic velocity laws}
    Under the assumptions of Proposition~\ref{proposition: stability - uncoupled problem}, there exists  $h_0>0$ such that the following error estimate holds for all $h\leq h_0$, uniformly for $0\le t \le T$:
    \begin{align*}
	    \|\ex(t)\|_{\bfK(\xs(t))}^2 + \|\ev(t)\|_{\bfM(\xs(t))}^2 + \int_0^t\!\! \|\ev(s)\|_{\bfA(\xs(s))}^2\, \d s 
	    \leq \! C \!\! \int_0^t \!\! \|\dv(s)\|_{\star,\xs}^2 \, \d s.
    \end{align*}
    The constant $C>0$ is independent of $t$ and $h$, but depends on the final time $T$ and the   parameter $\alpha$.
\end{proposition}
\ecl
\begin{proof}
By testing the error equation with $\ev$ we obtain
\begin{equation*}
    \begin{aligned}
        \ev^T \diff \Big(\bfM(\xs) \ev\Big) + \ev^T \bfA(\xs)\ev =
        &\ - \ev^T \diff \Big(\big(\bfM(\bfx) - \bfM(\xs)\big) \vs \Big) \\
        &\ - \ev^T \diff \Big(\big(\bfM(\bfx) - \bfM(\xs)\big) \ev \Big)\\
        &\ - \ev^T \big(\bfA(\bfx) - \bfA(\xs)\big) \vs  \\
        &\ - \ev^T \big(\bfA(\bfx) - \bfA(\xs)\big) \ev \\
        &\ + \ev^T \big(\bfg(\bfx)-\bfg(\xs)\big) - \ev^T \bfM(\xs)\dv .
    \end{aligned}
\end{equation*}
The terms  are bounded in the same way as the corresponding terms in the proofs of
Propositions~\ref{proposition: stability - uncoupled problem} and~\ref{proposition: stability - coupled problem}. With these estimates, a Gronwall inequality yields the result.
\QED \end{proof}

With Proposition \ref{proposition: stability - dynamic velocity laws} and the appropriate defect bounds,
Theorem~\ref{theorem: coupled error estimate} extends directly to the parabolic surface PDE coupled with the dynamic velocity law.

\section{Numerical results}
\label{section: numerics}

In this section we complement Theorem~\ref{theorem: coupled
  error estimate} by showing the numerical behaviour of piecewise linear finite elements, which are not covered by Theorem~\ref{theorem: coupled
  error estimate}, but nevertheless perform remarkably well.
   Moreover, we compare our regularized velocity law with regularization by mean curvature flow.  

\subsection{A coupled problem}
\label{sec:coupled-problem}

Our test problem is a combination of \eqref{eq: coupled problem}
with a mean curvature term as in \eqref{eq:reg_mcf}:
\begin{equation}
  \label{eq:2}
  \begin{aligned}
    \mat u + u \nbg \cdot v - \lb u  = {}
    & f(t,x), \\
    v - \alpha \lb v - \beta \lb X = {}
    & \delta u \normal_{\surface} + g(t,x)\nu_\Ga, \\
  \end{aligned}
\end{equation}
for non-negative parameters \(\alpha,\beta,\delta\). The velocity law here is  a special case of \eqref{eq: coupled problem} for \(\beta = 0\), and  reduces to \eqref{eq:reg_mcf} for $\delta=0$. The matrix-vector form reads 
\begin{align*}
  \dfdt{}{t}\Bigl( \mass\bigl(\nnodes(t)\bigr)\bfu(t)\Bigr) +
  \stiff\bigl(\nnodes(t)\bigr) \bfu(t)
  = {}
  & \bff\bigl(t,\nnodes(t)\bigr),
  && t\in [0,T], \\
  \regmass \bigl(\nnodes(t)\bigr) \dot{\bfx}(t) + \beta
  \stiff\bigl(\nnodes(t)\bigr) \nnodes(t)= {}
  & \delta \mathbf{N}\bigl(\nnodes(t))\bfu(t) +
    \bfg\bigl(t, \nnodes(t) \bigr),
  &&  t\in [0,T],
\end{align*}
for given \(\bfx(0)\) and \(\bfu(0)\), where 
\begin{equation*}
  \mathbf{N}\bigl(\nnodes)\bfu \rvert_{3(j-1)+\ell}
  = \int_{\surface_{h}(\nnodes)} \bigl(\normal_{\Ga_h}\bigr)_{\ell} \bfu_{j}
  \phi_{j}[\nnodes],
\end{equation*}
for \(j=1,\dotsc, N\) and \(\ell = 1,2,3\).  \par
In our numerical experiments we used a linearly implicit Euler discretization of this system with step sizes chosen so small that the error is dominated by the spatial discretization error.

\begin{example}
  \label{example:1}
  We consider \eqref{eq:2} and choose \(f\) and \(g\) such that \(X(p,t)= r(t)
  p\) with
  \begin{equation*}
    r(t) = \frac{r_{0}r_{K}}{r_{K}e^{-kt} + r_{0}(1-e^{-kt})}
  \end{equation*}
  and \(u(X,t)= X_{1}X_{2} e^{-6t}\) are the exact solution of the problem.  The
  parameters are set to be \(T = 1\), \(\alpha=1\), \(\beta=0\), \(\delta=0.4\),
  \(r_{0}= 1\), \(r_{K}=2\) and \(k = 0.5\).

  We choose $(\mathcal{T}_k)$ as a series of meshes  such that $2 h_k \approx h_{k-1}$.
  In Table~\ref{table:ex1} we report on the errors and the corresponding experimental orders of
  convergence (EOC). Using the notation of Section~\ref{section:lifts},
  the following norms are used:
  \begin{align*}
   \|{\rm err}_u \|_{ L^{\infty}(L^{2})} &= \sup_{[0,T]} \lVert \widehat u_{h}(\dotarg, t)
    - \Ih u(\dotarg, t) \rVert_{L^{2}(\surface_h^*(t))},  \\
    \|{\rm err}_u \|_{ L^{2}(H^1)}&= \left( \int_{0}^{T}
      \Bigl\lVert \widehat u_{h}(\dotarg,s)- \Ih u(\dotarg,s)
      \Bigr\rVert^{2}_{H^{1}(\surface_h^*(s))}
      \mathrm{d}s
    \right)^{\frac{1}{2}},  \\
    \|{\rm err}_v \|_{ L^{\infty}(H^1)} &= \sup_{[0,T]}\lVert \widehat v_{h}(\dotarg, t)
    - \Ih v(\dotarg, t) \rVert_{H^1(\surface_h^*(t))},  \\
    \|{\rm err}_x \|_{ L^{\infty}(H^1)} &= \sup_{[0,T]}\lVert \widehat x_{h}(\dotarg, t)
    - \id_{\surface_h^*(t)} \rVert_{H^1(\surface_h^*(t))} .
  \end{align*}
  The EOCs for the errors \(E(h_{k-1})\) and \(E(h_{k})\)
  with mesh sizes \(h_{k-1},h_{k}\) are given via
  \begin{equation*}
    EOC(h_{k-1},h_{k}) = \frac{\log
      \left(\frac{E(h_{k-1})}{E(h_{k})}\right)}{\log\left(
        \frac{h_{k-1}}{h_{k}}\right)} , \qquad (k=2,\dotsc,n).
  \end{equation*}
  The degree of freedoms (DOF) and maximum mesh size at time \(T\) are also
  reported in the tables.

  In Table~\ref{table:ex1} we report on the errors and EOCs observed using Example~\ref{example:1}. The EOCs in the PDE are expected to be $2$ for the $L^\infty(L^2)$ norm and $1$ for the $L^2(H^1)$ norm, while the errors in the surface and in  the surface velocity are expected to be $1$ in the \(L^{\infty}(H^{1})\) norm.
  \begin{table}[!h]
    \centering
    \subfloat[Errors for \(u\)]{
      \begin{tabular}{c rr ll ll}
        \toprule
        level & DOF & \(h(T)\)
        & $\|{\rm err}_u \|_{ L^{\infty}(L^{2})}$ & EOC
        & $ \|{\rm err}_u \|_{ L^{2}(H^1)}$ & EOC \\
        \midrule
        1 & 126   & 0.6664 &  0.1519165   & -    & 0.2727214 & -       \\
        2 & 516   & 0.4088 &  0.0896624   & 1.08 & 0.1498895 & 1.22 \\
        3 & 2070  & 0.1799 &  0.0222349   & 1.70 & 0.0344362 & 1.79 \\
        4 & 8208  & 0.0988 &  0.0070552   & 1.91 & 0.0109074 & 1.92 \\
        5 & 32682 & 0.0499 &  0.0018319   & 1.98 & 0.0029375 & 1.92 \\
        \bottomrule
      \end{tabular}
    }
    \\[2ex]
    \subfloat[Surface and velocity errors]{
      \begin{tabular}{c rr ll ll}
        \toprule
        level & DOF & \(h(T)\)
        & \(\|{\rm err}_v \|_{ L^{\infty}(H^1)}\) & EOC
        & \(\|{\rm err}_x \|_{ L^{\infty}(H^1)}\) & EOC \\
        \midrule
        1 & 126   & 0.6664 & 0.2260428  & -    & 0.1473157  & -       \\
        2 & 516   & 0.4088 & 0.0595755  & 2.73 & 0.0298673  & 3.27 \\
        3 & 2070  & 0.1799 & 0.0158342  & 1.61 & 0.0106836  & 1.25 \\
        4 & 8208  & 0.0988 & 0.0053584  & 1.81 & 0.0042312  & 1.54 \\
        5 & 32682 & 0.0499 & 0.0019341  & 1.50 & 0.0017838  & 1.27 \\
        \bottomrule
      \end{tabular}
    }
    \caption{Errors and EOCs for Example~\ref{example:1} }
    \label{table:ex1}
  \end{table}
\end{example}

\begin{example}
  \label{example:3}
  Again we consider \eqref{eq:2}, but this time we quantitatively compare
  the two different regularized velocity laws.  Hence, we let \(\delta\) vanish.  We use a \(g\) like in Example~\ref{example:1} and run two tests with the common
  parameters \(T = 2\), \(r_{0}= 1\), \(r_{K}=2\) and \(k = 0.5\), and use the
  same mesh and time step levels as before.  The first
  test uses \(\alpha=0\) and \(\beta=1\) and the second test uses
  \(\alpha=1\) and \(\beta=0\).  The results are captured in
  Table~\ref{table:ex3}.
  \begin{table}[!h]
    \centering
    \subfloat[Surface and velocity errors with parameters \(\alpha = 0 \) and \(\beta=1\).]{
      \begin{tabular}{c rr ll ll ll}
        \toprule
        level & DOF & \(h(T)\)
        &  \(L^{\infty}(L^{2})_{v}\) & EOC
        & \(L^{\infty}(H^{1})_{v}\) & EOC
        & \(L^{\infty}(H^{1})_{x}\) & EOC \\
        \midrule
        1 & 126   & 0.6664 & 0.756045  & -     & 1.31532  & -     & 1.601255  & -    \\
        2 & 516   & 0.4088 & 0.393067  & 1.34  & 0.78538  &  1.06 & 0.522342  & 2.29 \\
        3 & 2070  & 0.1799 & 0.095914  & 1.72  & 0.96206  & -0.25 & 0.137396  & 1.63 \\
        4 & 8208  & 0.0988 & 0.035166  & 1.67  & 1.48784  & -0.73 & 0.044666  & 1.87 \\
        5 & 32682 & 0.0499 & 0.019755  & 0.85  & 2.73584  & -0.89 & 0.013507  & 1.75 \\
        \bottomrule
      \end{tabular}
    }
    \\[2ex]
    \subfloat[Surface and velocity errors with parameters \(\alpha = 1\) and  \(\beta = 0\).]{
    \begin{tabular}{c rr ll ll ll}
      \toprule
      level & DOF & \(h(T)\)
      & \(L^{\infty}(L^{2})_{v}\) & EOC
      & \(L^{\infty}(H^{1})_{v}\) & EOC
      & \(L^{\infty}(H^{1})_{x}\) & EOC \\
      \midrule
      1 & 126   & 0.6664 & 0.149836   & -    & 0.225114 & -    & 0.143419 & -     \\
      2 & 516   & 0.4088 & 0.036118   & 2.91 & 0.058147 & 2.77 & 0.024087 & 3.65  \\
      3 & 2070  & 0.1799 & 0.009286   & 1.65 & 0.015843 & 1.58 & 0.009702 & 1.11  \\
      4 & 8208  & 0.0988 & 0.002705   & 2.06 & 0.005361 & 1.81 & 0.003990 & 1.48  \\
      5 & 32682 & 0.0499 & 0.000686   & 2.01 & 0.001935 & 1.49 & 0.001746 & 1.21  \\
      \bottomrule
    \end{tabular}
    }
    \caption{Errors and EOCs for Example~\ref{example:3}.}
    \label{table:ex3}
  \end{table}
  Our regularized velocity law provides smaller errors as regularizing with mean curvature flow.
  The EOCs in the errors in the surface and in the errors for the surface velocity are expected to be $1$ in  \(L^{\infty}(H^{1})_{v}\) and \(L^{\infty}(H^{1})_{x}\) norm, see Table~\ref{table:ex3}.b.
  While it can be observed that for this particular example the convergence rates for
  \(\alpha\neq 0\) are higher then for \(\beta \neq 0\).
\end{example}

\subsection{A model for tumor growth}
\label{sec:model-tumor-growth}

Our next test problem is the coupled system of equations
\begin{equation}
  \label{eq:3}
  \begin{aligned}
    \mat u + u \nbg \cdot v - \lb u  = {}
    & f_{1}(u,w), \\
    \mat w + w \nbg \cdot v - D_{c}\lb w  = {}
    & f_{2}(u,w), \\
    v - \alpha \lb v - \beta \lb X = {}
    & \delta u \normal_{\surface}, \\
  \end{aligned}
\end{equation}
where
\begin{equation*}
  f_{1}(u,w) = \gamma(a-u+u^{2}w), \quad
  f_{2}(u,w) = \gamma(b-u^{2}w),
\end{equation*}
with non-negative parameters \(D_{c}, \gamma, a, b, \alpha, \beta\). \par For
\(\alpha=0\) this system has been used as a simplified model for tumor growth; see Barreira, Elliott and Madzvamuse \cite{BEM} and \cite{ElliottStyles_ALEnumerics,CGG}.  These authors used the mean curvature term with a small
parameter \(\beta>0\) to regularize their velocity law.

We used piecewise linear finite elements and the same time discretization scheme as in \cite{BEM,ElliottStyles_ALEnumerics}.
\begin{example}
  \label{example:2}
  We consider \eqref{eq:3} and want to compare qualitatively the two different
  regularized velocity laws \(\alpha\neq 0\) and \(\beta\neq 0\).  As
  common parameters we use \(D_{c}= 10\), \(\gamma = 100\), \(a = 0.1\), \(b =
  0.9\) and \(T = 5\).  The initial surface is a sphere and the initial values
  \(u_{0}\) and \(w_{0}\) are calculated by solving an auxiliary surface PDE as
  follows.  We take small perturbations around the steady state
  \begin{equation*}
    \binom{\widetilde{u}_{0}}{\widetilde{w}_{0}} = \binom{a+b +
      \varepsilon_{1}(x)}{\frac{b}{(a+b)^{2}} + \varepsilon_{2}(x)},
  \end{equation*}
  where \(\varepsilon_{1}(x),\varepsilon_{2}(x)\in [0,0.01]\) take random
  values.  We solve the auxiliary coupled diffusion equations with the
  stationary initial
  surface until time \(\widetilde{T}=5\).  We set \(u_{0}=
  \widetilde{u}(\widetilde{T})\) and \(w_{0}= \widetilde{w}(\widetilde{T})\),
  which we used as initial values for \eqref{eq:3}.  \par
  We perform two experiments with \((\alpha,\beta) = (0,0.01)\) and
  \((\alpha, \beta) = (0.01, 0)\).  We present snapshots in
  Figure~\ref{fig:example2}.  We observe that both velocity laws display the
  same qualitative behavior, also agreeing with  \cite{ElliottStyles_ALEnumerics}.
  \begin{figure}
    \centering
    \subfloat[time \(t=0\)]{
      \includegraphics[width=.95\textwidth]{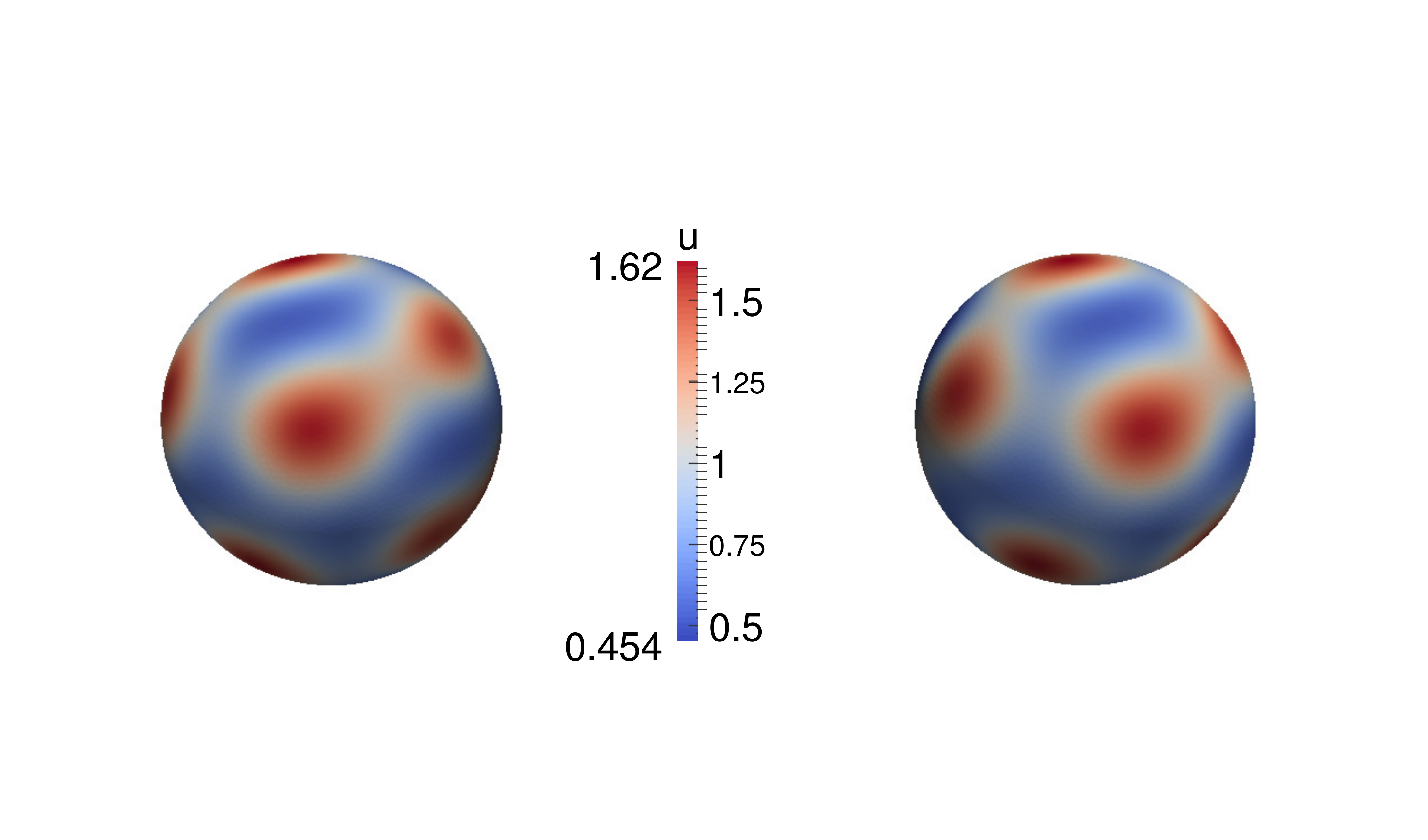}
    } \\
    \subfloat[time \(t=1\)]{
      \includegraphics[width=.95\textwidth]{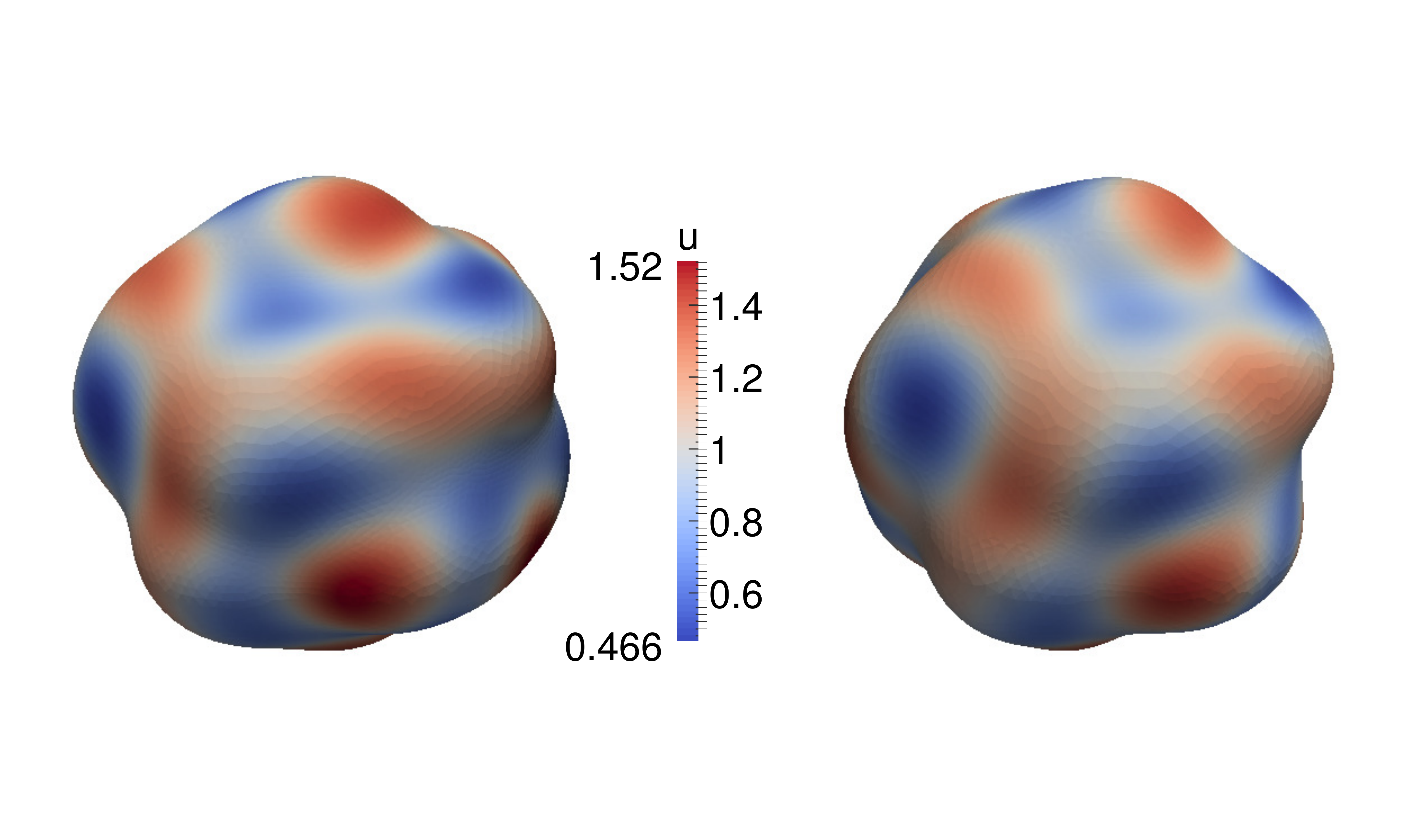}
    } \\
    \subfloat[time \(t=2\)]{
      \includegraphics[width=.95\textwidth]{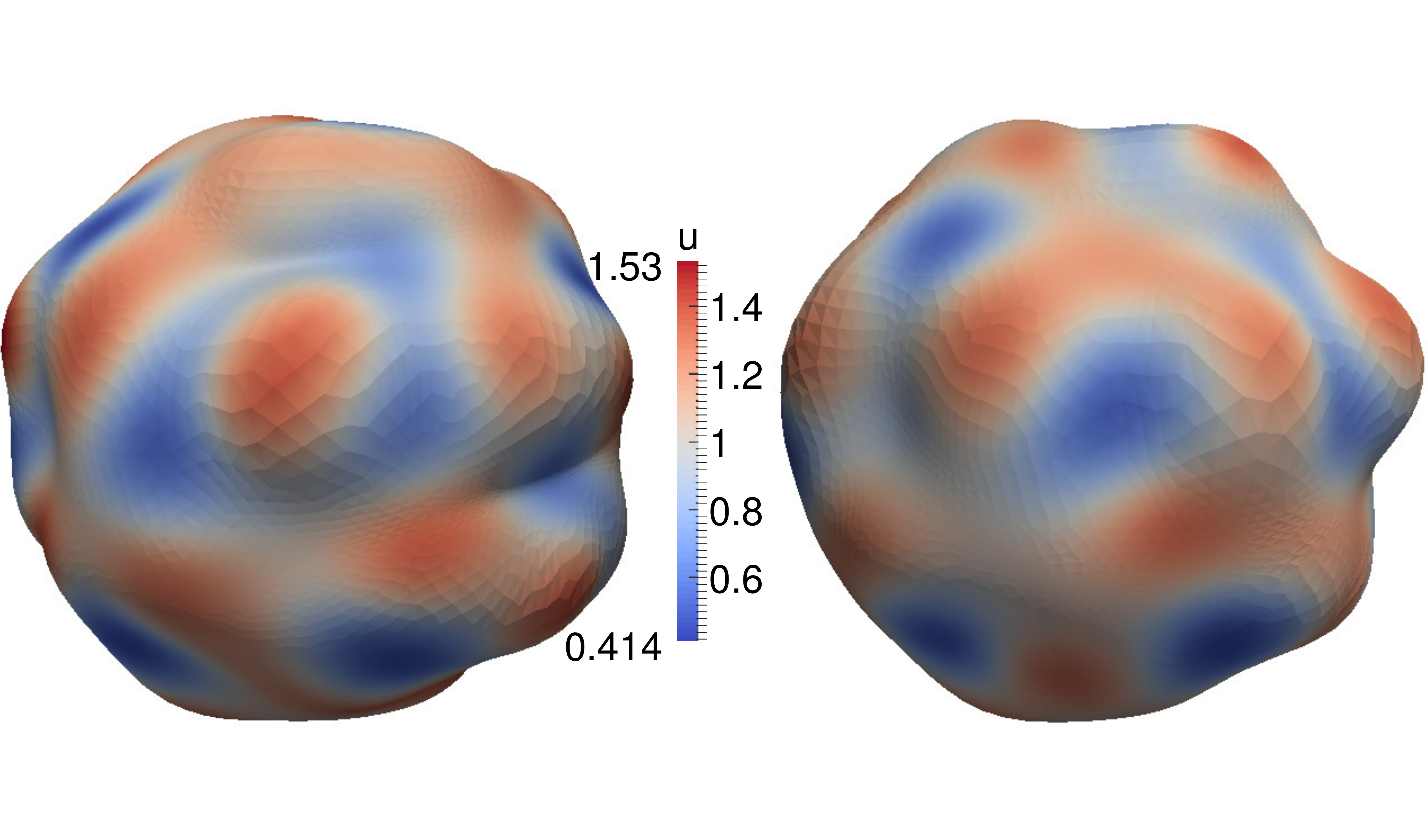}
    }
    \caption{Simulation for Example~\ref{example:2}.  The first column
      corresponds to \((\alpha, \beta) = (0, 0.01)\) and the second column to
      \((\alpha,\beta) = (0.01, 0)\).} 
    \label{fig:example2}
  \end{figure}
\end{example}
 
\section*{Acknowledgement}
\bbk \ebk
The research stay of Buyang Li at the University of T\"ubingen during the preparation of this paper was funded by the Alexander von Humboldt Foundation. The work of Buyang Li  was partially supported by a grant from the Germany/Hong Kong Joint Research Scheme sponsored by the Research Grants Council of Hong Kong and the German Academic Exchange Service of Germany (Ref. No. G-PolyU502/16). The work of Bal\'azs Kov\'acs is funded by Deutsche Forschungsgemeinschaft, SFB 1173.

We are grateful to Raquel Barreira and Antoni Madzvamuse for helpful comments on the tumor growth example of Section~\ref{sec:model-tumor-growth}. 
\bcl  We thank an anonymous referee for constructive comments on a previous version. \ecl

\end{document}